\documentclass[a4paper]{amsart}
\usepackage[utf8]{inputenc}
\usepackage[T1]{fontenc}
\usepackage{lmodern}
\usepackage{amssymb}
\usepackage[all]{xy}
\usepackage{nicefrac,mathtools,enumitem}
\usepackage{microtype}
\usepackage[pdftitle={Integrability of dual coactions},
  pdfauthor={Alcides Buss},
  pdfsubject={Mathematics}
]{hyperref}
\usepackage[lite]{amsrefs}
\newcommand*{\MRref}[2]{ \href{http://www.ams.org/mathscinet-getitem?mr=#1}{MR #1}}
\newcommand*{\arxiv}[1]{\href{http://www.arxiv.org/abs/#1}{arXiv: #1}}

\numberwithin{equation}{section}
\theoremstyle{plain}
\newtheorem{theorem}[equation]{Theorem}
\newtheorem{lemma}[equation]{Lemma}
\newtheorem{proposition}[equation]{Proposition}
\newtheorem{corollary}[equation]{Corollary}
\theoremstyle{definition}
\newtheorem{definition}[equation]{Definition}

\theoremstyle{remark}
\newtheorem{remark}[equation]{Remark}
\newtheorem{example}[equation]{Example}

\DeclareMathOperator{\spn}{span}
\DeclareMathOperator{\cspn}{\overline{span}}

\DeclareMathOperator{\aut}{Aut}
\DeclareMathOperator{\Dom}{dom}
\DeclareMathOperator{\Av}{Av}

\newcommand*{\cstar}{\texorpdfstring{$C^*$\nobreakdash-\hspace{0pt}}{C*-}}
\newcommand*{\Star}{\texorpdfstring{$^*$\nobreakdash-\hspace{0pt}}{*-}}

\newcommand*{\nb}{\nobreakdash}

\newcommand*{\defeq}{\mathrel{\vcentcolon=}}

\newcommand*{\<}{\langle}
\renewcommand{\>}{\rangle}

\newcommand*{\pc}{\vskip 1pc}

\newcommand*{\Real}{\mathbb{R}}

\newcommand*{\C}{\mathbb{C}}
\newcommand*{\Cl}{\mathcal{C}}

\newcommand*{\M}{\mathcal{M}}
\newcommand*{\MM}[1]{\mathcal{M}\bigl({#1}\bigr)}
\newcommand*{\tildeMM}[1]{\tilde{\mathcal{M}}\bigl({#1}\bigr)}

\newcommand*{\W}{\mathcal{W}}

\newcommand*{\E}{\mathcal{E}}
\newcommand*{\F}{\mathcal{F}}

\newcommand*{\G}{\mathcal{G}}

\newcommand*{\fell}{\mathcal{B}}
\newcommand*{\Ls}{\mathcal{L}}
\newcommand*{\K}{\mathcal{K}}

\newcommand*{\la}{\lambda}
\newcommand*{\f}{\varphi}

\newcommand*{\dualgcstar}[1]{$\dualg$\nb-\cstar{#1}}
\newcommand*{\cstaralgs}{\cstar{algebras}}
\newcommand*{\cstarfell}{C^*(\fell)}
\newcommand*{\cstarfellr}{C^*_\red(\fell)}

\newcommand*{\com}{\delta}
\newcommand*{\alg}{A}
\newcommand*{\algb}{B}
\newcommand*{\algc}{C}
\newcommand*{\hilm}{\mathcal{E}}
\newcommand*{\co}[1]{\delta_{#1}}
\newcommand*{\triv}{\mathrm{tr}}
\newcommand*{\repu}{\mathcal{U}}
\newcommand*{\fourier}[1]{A({#1})} 
\newcommand*{\fourst}[1]{B_\red({#1})} 
\newcommand*{\dtg}{\delta}
\newcommand*{\fmod}{\Delta} 
\newcommand*{\sbe}{\subseteq}
\newcommand*{\dd}{\,\mathrm{d}}
\newcommand*{\dualg}{\widehat{G}} 
\newcommand*{\crossp}[2]{C^*({#1},{#2})}
\newcommand*{\crosspr}[2]{C^*_\red({#1},{#2})}

\newcommand*{\braket}[2]{\langle#1\!\mid\!#2\rangle}

\newcommand*{\id}{{\mathrm{id}}}
\newcommand*{\su}{{\mathrm{su}}}
\newcommand*{\st}{{\mathrm{s}}}
\newcommand*{\si}{{\mathrm{si}}}
\newcommand*{\ii}{{\mathrm{i}}}
\newcommand*{\red}{\mathrm{r}} 
\newcommand*{\cont}{\mathcal{C}}
\newcommand*{\csg}{C_\red^*(G)}
\newcommand*{\csgf}{C^*(G)}
\newcommand*{\Mat}{\mathbb{M}} 

\begin{document}
\title[Integrability of dual coactions]{Integrability of dual coactions \\ on Fell bundle \cstar{algebras}}

\author{Alcides Buss}
\email{alcides@mtm.ufsc.br}
\address{Departamento de Matem\'atica\\
 Universidade Federal de Santa Catarina\\
 88.040-900 Florian\'opolis-SC\\
 Brazil}

\begin{abstract}
We study integrability for coactions of locally compact groups.
For abelian groups, this corresponds to integrability of the associated action of the Pontrjagin dual group.
The theory of integrable group actions has been previously studied by Ruy Exel, Ralf Meyer and Marc Rieffel.

Our goal is to study the close relationship between integrable group coactions and Fell bundles.
As a main result, we prove that dual coactions on \cstar{}algebras of Fell bundles are integrable,
generalizing results by Ruy Exel for abelian groups.
\end{abstract}
\subjclass[2000]{46L55 (46L08 81R50)}
\keywords{Fell bundles, integrable group coactions, dual coaction, noncommutative Fourier analysis,
Fourier inversion theorem, Plancherel weight}
\thanks{This work is based on the author's doctoral dissertation
    under the supervision of Ralf Meyer and Siegfried Echterhoff.}
\maketitle

\section{Introduction}

Fell bundles, also called \cstar{algebraic} bundles in \cite{Doran-Fell:Representations_2}, have attracted interest
of many researchers in operator algebras. They contain both theories of locally compact groups and \cstar{algebras}.
Moreover, \cstar{dynamical} systems can be viewed as special types of Fell bundles.

A Fell bundle over a locally compact group $G$ is a continuous Banach bundle $\fell=\{\fell_t\}_{t\in G}$
(see \cite{Doran-Fell:Representations_2}) endowed with a multiplication and an involution which is
compatible with the product and the inversion of $G$ in the sense that $\fell_t\cdot \fell_s\sbe \fell_{ts}$
and $\fell_t^*\sbe\fell_{t^{-1}}$ for all $t,s\in G$. There is a topology on $\fell$ which is in practical
situations most appropriately specified by a suitable space of continuous sections \cite[II.13.18]{Doran-Fell:Representations} and,
by definition, the algebraic operations on $\fell$ are required to be continuous with respect to this topology.
The axiom that connects Fell bundles to the theory of operator algebras
is the \cstar{identity} $\|b^*b\|=\|b\|^2$ which must hold for all $b\in \fell$.
In particular, $\fell_e$ is a \cstar{algebra} with respect to the restricted product and involution of $\fell$,
where $e$ is the identity element of $G$.

To a Fell bundle $\fell$ are associated two \cstar{algebras} in a natural way.
The \emph{full \cstar{algebra}} $\cstarfell$ which codifies the representation theory of $\fell$,
and the \emph{reduced \cstar{algebra}} $\cstarfellr$ which is the one generated by the regular representation of $\fell$.
Both \cstar{algebras} can be viewed are suitable completions of the \Star{algebra}
$\cont_c(\fell)$ of compactly supported continuous sections of $\fell$ endowed with the usual operations of convolution and involution.

If $G$ is discrete and abelian, there is a close relationship between Fell bundles over $G$
and actions of the Pontrjagin dual $\dualg$ on \cstar{algebras}, that is, \dualgcstar{algebras}.
Firstly, since $G$ is abelian, the full and the reduced \cstar{algebras} of $\fell$ are
isomorphic (the regular representation is faithful) and, what is more important, there is a natural action $\beta$
of $\dualg$ on $\cstarfell\cong\cstarfellr$, called the \emph{dual action}, which is determined by
\begin{equation}\label{eq:DefDualAction}
\beta_x(\xi)(t)\defeq \overline{\braket{x}{t}}\xi(t)\quad\mbox{for }\xi\in \cont_c(\fell), x\in \dualg, t\in G.
\end{equation}
Here we write $\braket{x}{t}\defeq x(t)$ to emphasize the duality between $G$ and $\dualg$.

Every action of $\dualg$ may be viewed as a dual action in a natural way if  $G$ is discrete, that is, if $\dualg$ is compact.
Indeed, given a \dualgcstar{algebra} $\alg$ with an action $\alpha$ of $\dualg$, we can define the \emph{spectral subspaces}:
\begin{equation}\label{eq:SpectralSubspaces}
\fell_t\defeq \left\{a\in \alg\colon\alpha_x(a)=\overline{\braket{x}{t}}a\mbox{ for all } x\in \dualg\right\}\quad\mbox{for }t\in G.
\end{equation}
Since $\dualg$ acts by \Star{}automorphisms, we have $\fell_t\cdot\fell_s\sbe\fell_{ts}$ and $\fell_t^*\sbe\fell_{t^{-1}}$ for all $t,s\in G$.
This implies that the collection $\fell=\{\fell_t\}_{t\in G}$ forms a Fell bundle with respect to the algebraic operations inherited from $\alg$.
The continuity of the bundle is irrelevant here because $G$ discrete. Note that the spectral subspaces are linear independent and
the (algebraic) direct sum $\bigoplus_{t\in G}\fell_t\sbe \alg$ can be canonically identified with the subspace $\cont_c(\fell)\sbe \cstarfell$.
Moreover, this identification induces an isomorphism $\alg\cong\cstarfell$ which is compatible with the actions of $\dualg$,
that is, it is an isomorphism of \dualgcstar{algebras}. Here we are implicitly using that $G$ is discrete to ensure that the
direct sum $\bigoplus_{t\in G}\fell_t$ be dense in $\alg$. In general, if $G$ is not discrete, the
spectral subspaces \eqref{eq:SpectralSubspaces} may be very trivial, for instance they may be all
zero even if $A$ is non-zero. This happens for example if $A$ carries a dual action.

As shown by Ruy Exel in \cite{Exel:Unconditional}, a necessary condition for a
\dualgcstar{algebra} to be isomorphic to a dual action is the integrability of the underlying action.
The notion of integrability of group actions was studied by Marc Rieffel in \cite{Rieffel:Integrable_proper} and is equivalent
to the one developed by Exel in \cite{Exel:Unconditional}. A positive element $a$
of a \dualgcstar{algebra} $\alg$ is called \emph{integrable} if the strict unconditional integral
$\int_{\dualg}^\su\alpha_x(a)\dd{x}$ exists, that is, if it converges unconditionally for the strict topology in $\M(\alg)$,
the multiplier algebra of $\alg$. Here we are tacitly using the Haar measure $\dd{x}$ of $\dualg$.
Linear combinations of integrable elements are also called integrable.
A \dualgcstar{algebra} $\alg$ is said to be \emph{integrable} if the set of integrable elements is dense in $\alg$.

Given an integral element $a$ in a \dualgcstar{algebra} $A$, we can define \emph{spectral elements}
\begin{equation}\label{eq:SpectralElements}
E_t(a)\defeq \int_{\dualg}^\su\braket{x}{t}\alpha_x(a)\dd{x}\quad\mbox{for }t\in G.
\end{equation}
Note that $E_t(a)$ is a kind of generalized Fourier coefficient of the element $a$. It belongs to the spectral subspace
\begin{equation}\label{eq:SpectralSubspacesMultiplier}
\M_t(A)\defeq \{a\in \M(A)\colon\alpha_x(a)=\overline{\braket{x}{t}}a\mbox{ for all }x\in \dualg\}.
\end{equation}
In the case of the dual action on the \cstar{algebra} $\cstarfell$ of a Fell bundle $\fell=\{\fell_t\}_{t\in G}$,
it was proved by Exel in \cite{Exel:Unconditional} that every element $\xi$ in
$\cont_c(\fell)^2\defeq \spn\{\eta*\zeta\colon\eta,\zeta\in \cont_c(\fell)\}$
is integrable (with respect to the dual action) and
\begin{equation}\label{eq:SpectralElementsDualAction}
E_t(\xi)=\xi(t)\quad\mbox{for all }t\in G,
\end{equation}
where we identify each $\xi(t)\in \fell_t$ as the multiplier in $\MM{\cstarfell}$ given by $\xi(t)\cdot\eta|_s\defeq \xi(t)\cdot\eta(t^{-1}s)$
whenever $\eta\in \cont_c(\fell)$ and $s\in G$.

What happens with these results if $G$ is non-abelian? The main goal of this work is to answer this question.
We are going to generalize Exel's results cited above to arbitrary locally compact groups.

If $G$ is not abelian, there is not a dual action on the \cstar{algebra} $\cstarfell$, but there is instead
a \emph{dual coaction} of $G$, that is, there is a nondegenerate \Star{}homomorphism
$$\co{\fell}\colon\cstarfell\to\MM{\cstarfell\otimes\csg}$$
which is compatible with the \emph{comultiplication} of $\csg$, the reduced \cstar{algebra} of $G$.
Here we implicitly consider $\csg$ as a \emph{locally compact quantum group}
in the sense of Johan Kustermans and Stefaan Vaes \cite{Kustermans-Vaes:LCQG}.
The comultiplication is the nondegenerate \Star{}homomorphism $\com\colon\csg\to\MM{\csg\otimes\csg}$
determined by the equation $\com(\la_t)=\la_t\otimes\la_t$
for $t\in G$, where $\la\colon G\to\Ls(L^2(G))$ denotes the left regular representation of $G$. The dual coaction is characterized by
$\co{\fell}(b_t)=b_t\otimes\la_t$ whenever $b_t\in \fell_t$. In the same way, there is a canonical $G$-coaction $\co{\fell}^\red$
on $\cstarfellr$, also called dual coaction, which makes the regular representation $\la_\fell\colon\cstarfell\to\cstarfellr$ equivariant,
that is, it is determined by the identity $\co{\fell}^\red(\la_\fell(b_t))=\la_\fell(b_t)\otimes\la_t$ for all $b_t\in \fell_t$.
In the case of abelian groups, coactions of $G$ correspond to actions of its dual $\dualg$
and the dual coaction corresponds to the dual action given by Equation~\eqref{eq:DefDualAction}.

Having defined the main objects, we can now enunciate our main task: we are going to prove that the dual coaction
on a \cstar{algebra} of a Fell bundle is integrable. The notion of integrable coactions of locally compact quantum groups
was defined and studied recently by the author and by Ralf Meyer in \cite{Buss-Meyer:Square-integrable}.
The main ingredient here is the Haar weight of the quantum group -- the noncommutative analogue of the Haar measure on a group.
The Haar weight of the quantum group $(\csg,\com)$ is the Plancherel weight \cite{Pedersen:CstarAlgebras,Takesaki:Theory_2}.
The weight theory generalizes the concept of measure to noncommutative operator algebras and the Plancherel weight
corresponds to the Haar measure on $\dualg$ if $G$ is abelian. Moreover, in this case integrable coactions of $G$
correspond to integrable actions of $\dualg$.

Although the theory of integrable coactions is sometimes more technical and less transparent than the theory of integrable actions
due to the use of weights in place of measures, many results have simpler proofs than the case of actions.
Our main result, that is, the fact that dual coactions are integrable, is an example of this phenomena.
Its proof not only generalizes, but also simplifies Exel's proof for abelian groups.

In the abelian case, the idea behind the proof is to view the dual action applied to an element $\xi\in \cont_c(\fell)$
as a kind of generalized Fourier transform of $\xi$. In this way, to prove that $\xi$
is integrable with respect to the dual action means the same as to prove that its Fourier transform is
strictly-unconditionally integrable. This is true if the function $\xi$ is positive-definite, and this holds provided
$\xi$ is positive as an element of $\cstarfell$. This proves that any element $\xi\in \cont_c(\fell)^2$ is integrable with
respect to the dual action. Moreover, Equation~\eqref{eq:SpectralElementsDualAction} is just
a manifestation of a (generalized) Fourier inversion Theorem.

We follow the same idea to proof that dual coactions are integrable in the case of non-abelian groups.
Basically, what we need is a generalized Fourier analysis for such groups.
In particular, we need a generalized Fourier inversion Theorem and this has been recently developed by the author in \cite{Buss:Fourier}.

\section{Preliminaries}

\subsection{Weight theory}\label{sec:weights}

In this section we recall some basic concepts of weight theory on \cstar{algebras}, mainly to fix our notation.
For more details, we refer the reader to~\cite{Kustermans-Vaes:Weight}.

In operator algebras, the noncommutative version of measure theory is the theory of weights.
Let us say that $X$ is a locally compact (Hausdorff) topological space and suppose that $\mu$ is a Radon measure on $X$.
Then we can associated to $\mu$ a function $\f_\mu\colon\cont_0(X)^+\to[0,\infty]$ defined by $\f_\mu(f)\defeq \int_X f(x)\dd{\mu(x)}$,
where $\cont_0(X)$ is the \cstar{algebra} of continuous functions on $X$ that vanish at infinity and $\cont_0(X)^+$
is the subset of positive functions in $\cont_0(X)$. Note that $\f_\mu$ is \emph{additive} and \emph{homogeneous}, that is,
$\f_\mu(\alpha f+g)=\alpha\f_\mu(f)+\f_\mu(g)$ for all $\alpha\in \Real^+$ and $f,g\in \cont_0(X)^+$.
The theory of weights recovers exactly this idea:

\begin{definition}
A \emph{weight} on a \cstar{algebra} $\algc$ is a function $\f\colon\algc^+\to [0,\infty]$
which is additive and homogeneous, where $\algc^+$ denotes the set of positive elements of $\algc$.
\end{definition}

Given a weight $\f$ on a \cstar{algebra} $\algc$, we define the following subsets of $\algc$:
the set of \emph{positive integrable} elements $\algc_\ii^+\defeq \{x\in \algc^+\colon\f(x)<\infty\}$,
the set of \emph{integrable} elements $\algc_\ii$ as the (complex) linear span of $\algc_\ii^+$,
and the set of \emph{square-integrable} elements $\algc_\si\defeq \{x\in \algc: x^*x\in \algc_\ii^+\}$.
Note that $\algc_\ii^+$ is a \emph{positive cone}, meaning that $\alpha x+ y\in \algc_\ii^+$
whenever $\alpha\in \Real^+$ and $x,y\in \algc_\ii^+$; also note that $\algc_\ii^+$ is \emph{hereditary}
in the sense that if $x,y\in \algc^+$ with $x\leq y$ and if $y\in \algc_\ii^+$, then $x\in \algc_\ii^+$.
This properties imply (see \cite[Proposition 2.6]{Exel:Spectral}) that $\algc_\ii$ is a
\Star{}subalgebra of $\algc$, that $\algc_\si$ is a left ideal of $\algc$, that $\algc_\ii$ is spanned
by $\algc_\si^*\algc_\si\defeq \{x^*y:x,y\in \algc_\si\}$, and that the positive part
of $\algc_\ii$ coincides with $\algc_\ii^+$. Moreover, the weight $\f$ extends linearly to $\algc_\ii$ and its extension
is also denoted by $\f$.

For the weight $\f_\mu$ on $\cont_0(X)$ defined above, $\cont_0(X)_\ii$
is the subalgebra of $\mu$\nb-integrable functions in $\cont_0(X)$, and
$\cont_0(X)_\si$ is the ideal of $\mu$\nb-square-integrable functions in $\cont_0(X)$.

To avoid pathological cases and allow good properties, we impose some natural conditions on the weights.
Note that for the weight $\f_\mu$, the subalgebra $\cont_0(X)_\ii$ is dense because
it contains the subset $\cont_c(X)$ of compactly supported functions.
In general, we say that a weight $\f$ on a \cstar{algebra} $\algc$ is
\emph{densely defined} if $\algc_\ii$ is dense in $\algc$. It is equivalent to require
that $\algc_\ii^+$ is dense in $\algc^+$, or that $\algc_\si$ is dense in $\algc$.
We are mainly interested in weights that are densely defined and \emph{lower semi-continuous}, that is,
\[
\f(x)\leq \lim\limits_{i\in I}\inf \f(x_i)
\]
whenever $(x_i)_{i\in I}$ is a net in $\algc^+$ converging to $x\in \algc^+$.
For the weight $\f_\mu$ defined from a Radon measure $\mu$, lower semi-continuity follows from Fatou's Lemma.

An important property about lower semi-continuous weights is that they can be approximated by \emph{bounded} linear functionals in a
natural way. More precisely, we have (see \cite{Combes:Poids})
\begin{equation}\label{eq:ApproxLowerSemiContinuousWeights}
\f(x)=\sup\{\omega(x)\colon\omega\in \F_\f\},
\end{equation}
where $\F_\f\defeq \{\omega\in \algc^*_+\colon\omega(x)\leq\f(x) \mbox{ for all }x\in \algc^+\}$.
Here $\algc^*_+$ denotes the set of all (bounded) positive linear functionals on $\algc$. Moreover, defining
$\G_\f\defeq \{\alpha\cdot\omega\colon\alpha\in(0,1), \omega\in \F_\f\}$, it is possible to prove that $\G_\f$ is a directed set with
respect to the natural order of $\algc^*_+$, that is,
for all $\omega_1,\omega_2\in \G_\f$, there is  $\omega\in \G_\f$ such that $\omega_1,\omega_2\leq\omega$
(see \cite{Kustermans:KMS} for the proof of this fact). This allows us to use $\G_\f$ as the index set of a net.
It follows from Equation~\eqref{eq:ApproxLowerSemiContinuousWeights} that
\begin{equation}\label{eq:NetApproxLowerSemiContinuousWeights}
\f(x)=\lim\limits_{\omega\in \G_\f}\omega(x)\quad\mbox{for all }x\in \M(C)_\ii.
\end{equation}
It is always possible to extend a lower semi-continuous weight $\f$ to the multiplier algebra
$\M(\algc)$ of $\algc$ in a natural way. In fact, we can define
$$\bar\f(x)\defeq \sup\{\bar\omega(x)\colon\omega\in \F_\f\}\quad \mbox{for all }x\in \M(\algc)^+,$$
where $\bar\omega$ denotes the strictly continuous extension of the functional $\omega\in \algc^*$.
Note that $\bar\f$ is the unique strictly lower semi-continuous weight on $\M(\algc)$ extending $\f$.
Abusing the notation, we shall also write $\f$ for the extension $\bar\f$ and use the notations $\M(\algc)_\ii$ and $\M(\algc)_\si$
for the subsets of $\M(\alg)$ of integrable and square-integrable elements with respect to $\bar\f$.
We do the same for functionals $\omega\in \algc^*$, that is, we also write $\omega$ for its unique strictly continuous extension
$\bar\omega$ on $\M(\algc)$.

\subsection{Slicing with weights}\label{sec:WeightSlices}

Let $\f$ be a lower semi-continuous weight on a \cstar{algebra} $\algc$.
Given a \cstar{algebra} $\alg$, it is possible to extend $\f$ naturally to a "generalized weight" $\id_\alg\otimes \f$
on $\alg\otimes \algc$ taking values in $\alg$. Here and throughout the rest of this work, the symbol $\otimes$ denotes
the \emph{minimal} tensor products between \cstar{algebras}. It is also possible to extend
$\id_\alg\otimes \f$ to multiplier algebras. Firstly, we define the domain of $\id_\alg\otimes\f$
as the set $\M(\alg\otimes \algc)_\ii^+$ of all $x$ in $\M(\alg\otimes \algc)^+$ for which the net
$\bigl((\id_\alg\otimes\omega)(x)\bigr)_{\omega\in \G_\f}$ converges in $\M^\st(\alg)$.
Here $\M^\st(\alg)$ denotes the multiplier algebra $\M(\alg)$ endowed with the strict topology, that is, the
locally convex topology generated by the semi-norms $a\mapsto \|ab\|$ and $a\mapsto \|ba\|$ with $b\in A$.
For each bounded functional $\omega\in \algc^*$, the \emph{slice map} $\id_\alg\otimes\omega\colon\alg\otimes \algc\to \algc$
is defined as the unique bounded linear map satisfying $(\id_\alg\otimes\omega)(a\otimes c)=a\omega(c)$
for all $a\in \alg$ and $c\in \algc$. Moreover, this map admits a unique strictly continuous extension
$\M(\alg\otimes \algc)\to\M(\alg)$ which is also denoted by $\id_\alg\otimes\omega$. This is the map we used above.
The slice map $\id_\alg\otimes\f\colon\M(\alg\otimes \algc)_\ii^+\to\M(\alg)$ is defined in the natural way:
\[
(\id_\alg\otimes\f)(x)\defeq \mbox{s-}\!\lim\limits_{\omega\in\G_\f}(\id_\alg\otimes\omega)(x)\quad\mbox{for all }
x\in \M(\alg\otimes \algc)_\ii^+.
\]
The script "s" serves to remember that we take the limit in the strict topology of $\M(\alg)$.
Note that if $\alg=\C$, then $\M(\alg\otimes \algc)_\ii^+=\M(\algc)_\ii^+$. As in this special case,
one has that $\M(\alg\otimes \algc)_\ii^+$ is a hereditary positive cone in $\M(\alg\otimes \algc)^+$.
The linear span of $\M(\alg\otimes \algc)_\ii^+$, denoted by $\M(\alg\otimes \algc)_\ii$,
is a \Star{}subalgebra of $\M(\alg\otimes \algc)$ whose positive part coincides with $\M(\alg\otimes \algc)_\ii^+$.
As for $\M(\algc)_\ii$, elements in $\M(\alg\otimes \algc)_\ii$
also called  \emph{integrable} (with respect to the weight $\f$).
Moreover, the set of \emph{square-integrable} elements
$\M(\alg\otimes \algc)_\si\defeq \{x\in \M(\alg\otimes \algc):x^*x\in \M(\alg\otimes \algc)_\ii\}$
is a left ideal in $\M(\alg\otimes \algc)$ and
\[
\M(\alg\otimes \algc)_\ii=\spn\M(\alg\otimes \algc)_\si^*\M(\alg\otimes \algc)_\si.
\]
The map $\id_\alg\otimes\f$ has a linear extension to $\M(\alg\otimes \algc)_\ii$, also denoted by $\id_\alg\otimes\f$.
More details about the facts mentioned above can be found in \cite{Kustermans-Vaes:Weight}.
In particular, the following result (Propositions 3.9 and 3.14 in \cite{Kustermans-Vaes:Weight}) characterizes integrable elements.

\begin{proposition}\label{prop:chacterization of integrable elements}
A positive element $x\in \M(\alg\otimes \algc)^+$ is integrable \textup(with respect to the weight $\f$ on $\algc$\textup)
if and only if there is $a\in \M(\alg)$ such that, for all positive linear functionals $\theta\in \alg^*_+$, the element
$(\theta\otimes\id_\algc)(x)$ is integrable in $\M(\algc)$ \textup(that is, it belongs to $\M(\algc)_\ii$\textup) and one has
$\f\bigl((\theta\otimes\id_\algc)(x)\bigr)=\theta(a)$.
\end{proposition}

It is interesting to describe what happens in the commutative case. Assume that $\algc=\cont_0(X)$
for some locally compact space $X$ and that $\f=\f_\mu$ is the weight associated to a Radon measure $\mu$ on $X$
(see Section~\ref{sec:weights}). In this case, $\M(\alg\otimes \algc)$ can be canonically identified with the \cstar{algebra}
$\cont_b\bigl(X,\M^\st(\alg)\bigr)$ of all bounded strictly continuous functions $f:X\to\M(\alg)$.\footnote{The script $\st$ in
$\cont_b\bigl(X,\M^\st(\alg)\bigr)$ is used to remember the use of strict topology in $\M(\alg)$.}
The concept of integrability defined above for elements in $\M(\alg\otimes \algc)$
can be translated in more classical notions of integrability (see \cite[Proposition 2.2]{Buss:Fourier}).
To be more precise, a function $f:X\to \M(\alg)$
which corresponds to a positive element in $\M(\alg\otimes \algc)$ is integrable
in the above sense if and only if $f$ is \emph{strictly-unconditionally integrable}, that is,
if the net of strict Bochner integrals $\bigl(\int_K^\st f(x)\dd{\mu(x)}\bigr)_{K\in \Cl}$ converges in the strict topology of $\M(\alg)$,
where $\Cl$ denotes the set of compact subsets of $X$ directed by inclusion.
Note that because $f:X\to \M(\alg)$ is strictly continuous, its strict Bochner integral exists over compact subsets
$K\sbe X$, that is, the functions $x\mapsto f(x)a$ and $x\mapsto af(x)$ are Bochner integrable over $K$ for all $a\in \alg$.

Unconditional integrability has been studied by Ruy Exel in \cite{Exel:Unconditional,Exel:Spectral}.
As observed by him in \cite{Exel:Spectral}, this concept of integrability is equivalent
to the notion of integrability in the sense of Pettis \cite{Pettis:Integration}.
In fact, this result holds not only for continuous functions over topological spaces with values in operator algebras,
but also in the more general context of measurable functions over measure spaces with values in arbitrary Banach spaces.
This is the main result of the dissertation of Patricia Hess \cite[Theorem 4.14]{Hess:Integracao}.
As we have shown in \cite[Proposition 2.2]{Buss:Fourier}, in the special case of operator algebras the proof can be
simplified significantly.

\subsection{The Plancherel weight}\label{533}
In this section, we recall some basic facts about the Plancherel weight on the reduced
\cstar{algebra} $C_\red^*(G)$ of a locally compact group $G$.
This is a fundamental ingredient for our future constructions.

We start by recalling the definition of the Plancherel weight $\tilde{\f}$ on the von Neumann
algebra $\Ls(G)\defeq C_\red^*(G)''\sbe\Ls(L^2(G))$ $\bigl($that is, the bicommutant of
$C^*_\red(G)$ in $\Ls(L^2(G))\bigr)$. For more details we indicate \cite[Section~7.2]{Pedersen:CstarAlgebras} or
\cite[Section~VII.3]{Takesaki:Theory_2}.

A function $\xi\in L^2(G)$ is called \emph{left bounded} if the map
$\cont_c(G)\ni f\mapsto \xi*f\in L^2(G)$ extends to a bounded operator on $L^2(G)$ and,
in this case, we denote its extension by $\la(\xi)$. Here $(\xi*f)(t)=\int_G\xi(s)f(s^{-1}t)\dd{s}$
denotes the convolution of functions. Note that $\la(\xi)$ belongs to $\Ls(G)$ for every left bounded function $\xi$.
The Plancherel weight $\tilde\f\colon\Ls(G)^+\to [0,\infty]$ is defined by the formula
$$
\tilde\f(x)\defeq \left\{
\begin{array}{cc}
\|\xi\|^2=\braket{\xi}{\xi} & \mbox{if }x^{\frac{1}{2}}=\lambda(\xi) \mbox{ for some left bounded function }\xi\in L^2(G),\\
\infty & \mbox{otherwise}.\qquad\qquad\qquad\qquad\qquad\qquad\qquad\qquad\qquad\quad\,\,\,\,
\end{array}\right.
$$
From the above definition it follows that
$$\Ls(G)_\si=\bigl\{\la(\xi)\colon\xi\in L^2(G)\mbox{ is left bounded}\bigr\}$$
and (by polarization) $\tilde\f\bigl(\la(\xi)^*\la(\eta)\bigr)=\braket{\xi}{\eta}$ whenever $\xi,\eta\in L^2(G)$ are left bounded.
Here $\<\xi|\eta\>=\int_G\overline{\xi(t)}\eta(t)\dd{t}$ denotes the inner product in $L^2(G)$ (which is assumed
to be linear on the second variable).

The Plancherel weight $\f$ on $C_\red^*(G)$ is, by definition, the restriction of $\tilde\f$ to $C_\red^*(G)^+$. Thus
$$\csg_\si=\bigl\{\la(\xi)\colon\xi\in L^2(G)\mbox{ is left bounded and }\la(\xi)\in \csg\bigr\}.$$
We also get
\begin{equation}\label{520}
\MM{\csg}_\si=\bigl\{\la(\xi)\colon\xi\in L^2(G)\mbox{ is left bounded and }\la(\xi)\in \MM{\csg}\bigr\}.
\end{equation}
Defining the usual involution $\xi^*(t)\defeq \Delta(t)^{-1}\overline{\xi(t^{-1})}$, where $\Delta$ is the modular function of $G$,
it is easy to see that
\begin{equation}\label{371}
(\xi^**\eta)(t)=\<\xi|V_t\eta\>\quad\mbox{for all }\xi,\eta\in L^2(G)\mbox{ and } t\in G,
\end{equation}
where $V_t(\eta)(s)\defeq \eta(st)$. In particular, the function $\xi^**\eta$ is continuous and we have
$(\xi^**\eta)(e)=\<\xi|\eta\>$, where $e$
denotes the identity element of $G$. Thus, if $\xi,\eta\in L^2(G)$ are left bounded functions, the operator
$\la(\xi^**\eta)=\la(\xi)^*\la(\eta)$ belongs to $\Ls(G)_\ii$ and
$$\tilde{\f}\bigl(\la(\xi^**\eta)\bigr)=\braket{\xi}{\eta}=(\xi^**\eta)(e).$$
Therefore $$\Ls(G)_\ii=\la\bigl(\cont_e(G)\bigr),$$
where $\cont_e(G)\defeq \spn\{\xi^**\eta\colon\xi,\eta\in L^2(G)\mbox{ left bounded}\}\sbe \cont(G)$,
and $\tilde\f$ is simply the functional that evaluates functions of $\cont_e(G)$ at $e\in G$.
Since $\f$ is the restriction of $\tilde\f$ to $\csg$, we have $\MM{\csg}_\ii\sbe\Ls(G)_\ii$ and the same formula holds for $\f$.
Note that Equation~\eqref{371} yields $\fmod(t)^{\frac{1}{2}}(\xi^**\eta)(t)=\braket{\xi}{\rho_t\eta}$, where
$\rho_t=\fmod(t)^{\frac{1}{2}}V_t$ is the right regular representation of $G$. It follows that
$\fmod^{\frac{1}{2}}\cdot(\xi^**\eta)\in \fourier{G}$, the Fourier algebra of the group $G$ (see \cite{Eymard:Fourier}).
In particular, $\fmod^{\frac{1}{2}}\cdot\cont_e(G)\sbe \fourier{G}.$
This inclusion has dense image in $\fourier{G}$ because $\cont_e(G)$ contains all the functions in
$\cont_c(G)^2=\spn \bigl(\cont_c(G)*\cont_c(G)\bigr)$.

Although the formula $\tilde\f\bigl(\la(f)\bigr)=f(e)$ makes sense for every function
$f$ in $\cont_c(G)$, it is not true in general that $\la\bigl(\cont_c(G)\bigr)\sbe \Ls(G)_\ii$,
that is, it is not true that $\cont_c(G)\sbe \cont_e(G)$ (of course, we always have
$\cont_c(G)*\cont_c(G)\sbe \cont_e(G)$). See Remark~4.5 in \cite{Buss:Fourier}.
However,  one can prove the following partial result (see Proposition 4.4 in \cite{Buss:Fourier}):

\begin{proposition}\label{349} Let $G$ be a locally compact group and let $f\in \cont_c(G)$.
If $\la(f)\geq 0$ as an operator on $L^2(G)$, then there is a left bounded function $\xi\in L^2(G)$
such that $\la(f)^{\frac{1}{2}}=\la(\xi)$ and $f=\xi^**\xi$. In particular,
$\la(f)\in\csg_\ii^+$ and $$\f\bigl(\la(f)\bigr)=\|\xi\|_2^2=(\xi^**\xi)(e)=f(e).$$
\end{proposition}

Finally, let us observe that $\tilde\f$ is a KMS weight (see \cite{Kustermans:KMS} for the definition of KMS weights).
The modular group of automorphisms $\{\sigma_x\}_{x\in \Real}$ associated to $\tilde\f$ is given by
\begin{equation}\label{370}
\sigma_x(a)=\nabla^{\\i x}a\nabla^{-\\i x},\quad a\in \Ls(G),\,\,x\in \Real,
\end{equation}
where $\nabla$ is the modular operator (in general unbounded, but always closed and densely defined). It is given by
$(\nabla\xi)(t)=\fmod(t)\xi(t)$ for $\xi\in \Dom(\nabla)\sbe L^2(G)$ and $t\in G$, where
the domain of $\nabla$ is
$$\Dom(\nabla)=\left\{\xi\in L^2(G)\colon\int_G|\xi(t)|^2\fmod(t)^2\dd{t}<\infty\right\}.$$
Equation~\eqref{370} implies that $\sigma_x(\la_t)=\fmod(t)^{\ii x}\la_t$ for all $t\in G$ and $x\in \Real$.
This implies that $\la_t$ is \emph{analytic} with respect to $\sigma$, that is, the function $x\mapsto \sigma_x(\la_t)$
extends to an analytic function on $\C$. Its extension is given by
\begin{equation}\label{564}
\sigma_z(\la_t)=\fmod(t)^{\ii z}\la_t\quad\mbox{for all }z\in \C,\, t\in G.
\end{equation}

\subsection{Generalized Fourier analysis}
Considering the Plancherel weight $\f$ as a generalization of the Haar measure on $\dualg$ for abelian groups,
one may wonder whether it is possible to use $\f$ to develop a Fourier analysis to non-abelian groups.
In this section, we describe some steps in this direction. More specifically, we define generalized Fourier transforms and describe
how a generalized Fourier inversion formula can be obtained. More details can be found in \cite{Buss:Fourier}.

\begin{definition}\label{367} Given $x\in \Ls(G)_\ii$, we define the \emph{Fourier transform}
of $x$ as the function $\hat{x}\colon G\to\C$ given by
$$\hat{x}(t)\defeq \tilde\f(\la_t^{-1}x),\quad t\in G.$$
\end{definition}

Note that, by definition, $\tilde\f(x)=\hat{x}(e)$ for all $x\in \Ls(G)_\ii$.
As already observed, $\la_t^{-1}=\la_{t^{-1}}$ is analytic with respect to the modular group of automorphisms of
$\tilde\f$ for all $t\in G$ (see Equation~\eqref{564}).
This implies that $\la_t^{-1}x\in \Ls(G)_\ii$ whenever $x\in \Ls(G)_\ii$.
Thus the Fourier transform is well-defined.

If $G$ is abelian, then, under the identification $\Ls(G)\cong L^\infty(\dualg)$, the Plancherel weight on $\Ls(G)$
corresponds to the usual Haar integral on $L^\infty(\dualg)$. Moreover, in this case
$\Ls(G)_\ii$ corresponds to the subalgebra $L^\infty(\dualg)\cap L^1(\dualg)\sbe L^\infty(\dualg)$ and the function
$\hat x$ corresponds to the Fourier transform of the function $f\in L^\infty(\dualg)\cap L^1(\dualg)$ associated to $x$, that is,
to the function $t\mapsto\hat{f}(t)\defeq \int_{\dualg}\braket{\chi}{t} f(\chi)\dd{\chi}$.
Proposition 2.5 in \cite{Buss:Fourier} describes some basic properties of the Fourier transform.
In particular, it says that $\hat{x}$ belongs to $\cont_e(G)$ for all $x\in \Ls(G)_\ii$.
In particular, $\hat{x}$ is continuous. Moreover, the Fourier transform $\la(f)$ is equal to $f$ for every function $f\in \cont_e(G)$.
If $\cont_e(G)$ is equipped with the usual operation of convolution and the involution
$f^*(t)\defeq \fmod(t^{-1})\overline{f(t^{-1})}$, then $\cont_e(G)$ is a \Star{}algebra and the map
$$\Ls(G)_\ii\ni x\mapsto \hat{x}\in \cont_e(G)$$
is an isomorphism of \Star{}algebras. The inverse map is given by $f\mapsto \la(f)$. In particular, we have
$$\widehat{xy}=\hat{x}*\hat{y},\quad\mbox{e}\quad \widehat{x^*}=\hat{x}^*\quad\mbox{for all }x,y\in\Ls(G)_\ii.$$

In what follows, we generalize the constructions above to allow operator valued Fourier transforms
using the slice map $\id_\alg\otimes\f\colon\M(\alg\otimes \csg)_\ii\to \M(\alg)$ defined in Section~\ref{sec:WeightSlices},
where $\alg$ is an arbitrary \cstar{algebra}.

\begin{definition}\label{FourierCoefficient} Let $\alg$ be a \cstar{algebra} and let $a\in \MM{\alg\otimes\csg}$ be an integrable element.
The \emph{Fourier coefficient} of $a$ at $t\in G$ is the element $\hat{a}(t)\in \M(\alg)$ defined by
$$\hat{a}(t)\defeq (\id\otimes\f)\bigl((1\otimes\la_t^{-1})a\bigr).$$
The map $t\mapsto \hat{a}(t)$ from $G$ to $\M(\alg)$ is the \emph{\textup(generalized\textup) Fourier transform} of $a$.
\end{definition}

As already observed, $\la_t^{-1}=\la_{t^{-1}}$ is an analytic element for all $t\in G$.
This implies that $(1_\alg\otimes\la_t^{-1})\cdot\MM{\alg\otimes\csg}_\ii\sbe\MM{\alg\otimes\csg}_\ii$
(see \cite[Proposition 3.28]{Kustermans-Vaes:Weight}), so that the Fourier transform is well-defined.

Suppose that the group $G$ is abelian. As already observed, there is a canonical isomorphism
$\M\bigl(\alg\otimes \csg\bigr)\cong\cont_b\bigl(\dualg,\M^{\st}(\alg)\bigr)$, the \cstar{algebra}
of bounded strictly continuous functions $\dualg\to\M(\alg)$.
Under this identification, the element $\bigl((1\otimes\la_{t^{-1}})a\bigr)$
corresponds to the function $x\mapsto \braket{x}{t}a(x)$. We also observed in Section~\ref{sec:WeightSlices} that an element
$a$ in $\MM{\alg\otimes\csg}^+$ is integrable if and only if the corresponding function $x\mapsto a(x)$ in
$\cont_b\bigl(\dualg,\M^\st(\alg)\bigr)$ is strictly-unconditionally integrable. Moreover,
in this case $(\id\otimes\f)(a)$ coincides with $\int_{\dualg}^\su a(x)\dd{x}$, where
the symbol $\int^\su$ represents the strict unconditional integral.
Recall that it is defined as the strict limit of the net $\bigl(\int_K^\st a(x)\dd{x}\bigr)_{K\in \Cl}$
of strict Bochner integrals, where $\Cl$ denotes the set of all compact subsets
of $\dualg$ directed by inclusion.

We conclude that the Fourier transform of an integrable element $a$ in the space $\MM{\alg\otimes\csg}\cong\cont_b\bigl(\dualg,\M^\st(\alg)\bigr)$
coincides with the Fourier transform defined by Exel in \cite{Exel:Unconditional}:
$$\hat{a}(t)=\int_{\dualg}^\su\braket{x}{t} a(x)\dd{x}.$$

Finally, we mention the main result in \cite{Buss:Fourier}.

\begin{theorem}[Fourier's inversion Theorem]\label{the:FourierInversionTheorem}
Let $G$ be a locally compact group and let $\alg$ be a \cstar{algebra}.
Let $a\in \M\bigl(\alg\otimes \csg\bigr)$ be an integrable element and
suppose that the function $G\ni t\mapsto \hat{a}(t)\otimes\la_t\in \M\bigl(\alg\otimes \csg\bigr)$
is strictly-unconditionally integrable. Then
$$a=\int_G^\su \hat{a}(t)\otimes\la_t\dd{t}.$$
\end{theorem}

Theorem~\ref{the:FourierInversionTheorem} extends to non-abelian groups the version of the Fourier inversion Theorem
obtained by Exel in \cite{Exel:Unconditional} for abelian groups.
In fact, suppose that $G$ is abelian. Then, through the identification
$\M\bigl(\alg\otimes \csg\bigr)\cong \cont_b\bigl(\dualg,\M^{\st}(\alg)\bigr)$,
the element $\hat{a}(t)\otimes\la_t$ corresponds to the function
$x\mapsto \overline{\braket{x}{t}}\hat{a}(t)$. Thus, Theorem~\ref{the:FourierInversionTheorem} says that
$$\int_G^\su\overline{\braket{x}{t}} \hat{a}(t)\dd{t}=a(x)$$
whenever $a$ is integrable and the strict unconditional integral above exists.
In this case, the Fourier transform $\hat{a}$ is
given by $\hat{a}(t)=\int_{\dualg}^\su\braket{\eta}{t}a(\eta)\dd{\eta}$.
Thus we may also rewrite the above equation in the form of a generalized Fourier inversion formula:
$$\int_G^\su\overline{\braket{x}{t}} \left(\int_{\dualg}^\su\braket{y}{t}a(y)\dd{y}\right)\dd{t}=a(x).$$
Exel's version of Fourier's inversion Theorem starts with a positive-definite, compactly supported, strictly continuous function
$f\colon G\to\M(\alg)$. Apparently, our version does not require any positivity condition on the functions involved. However, we
implicitly have such a condition  because integrable elements are defined in terms of positive elements.

To obtain a more precise relation between our version of  Fourier's inversion Theorem and Exel's one,
let us first recall that a function $f\colon G\to\M(\alg)$ is \emph{positive-definite} if,
for every finite subset $\{t_1,\ldots,t_n\}$ of $G$, the matrix $\bigl(f(t_i^{-1}t_j)\bigr)_{i,j}$ is positive as an element of
the \cstar{algebra} $\Mat_n\bigl(\M(\alg)\bigr)$ of all $n\times n$ matrices with entries in $\M(\alg)$.

A strictly continuous function $f\colon G\to\M(\alg)$ is positive-definite if and only if it has the form $f(t)=T^*w_tT$
for some strongly continuous unitary representation $w\colon G\to \Ls(\hilm)$ on some Hilbert $\alg$-module $\E$ and
$T\colon\alg\to\hilm$ is some adjointable operator. Basically, this result is a generalized version of Naimark's Theorem on
the structure of positive-definite functions (\cite{Neumark:PositiveDefinite,Paulsen:CompletelyBoundedMaps}).
For a detailed proof of this fact see \cite[Proposition 4.2]{Buss:Fourier}.
This characterization implies that such functions are automatically bounded and left uniformly continuous
in the strong topology of $\M(\alg)$, that is, for all $\xi\in \alg$, $\|f(ts)\xi- f(t)\xi\|$ converges
to zero uniformly in $t$ as $s$ converges to $e$ (the identity of the group $G$).

Another characterization of positive-definite functions is given in \cite{Buss:Fourier} using the left regular representation
$\la:C_c(G)\to\Ls(L^2(G))$ -- defined by $\la(f)\xi=f*\xi$ for all $f\in C_c(G)$ and $\xi\in L^2(G)$. To describe this,
we need to extend $\la$ to the space $\cont_c(G,\M^\st(\alg))$ of compactly supported strictly continuous functions $f\colon G\to\M(\alg)$, where
$\alg$ is a given \cstar{algebra}. Note that $\cont_c\bigl(G,\M^\st(\alg)\bigr)$ is a \Star{}algebra with the usual operations of
convolution and involution (see \cite[Proposition C.6]{Echterhoff-Kaliszewski-Quigg-Raeburn:Categorical}):
$$f*g(t)\defeq \int_G^\st f(s)g(s^{-1}t)\dd{t}\quad\mbox{and}\quad f^*(t)\defeq \Delta(t^{-1})f(t^{-1})^*$$
for all $f,g\in\cont_c(G,\M^\st(\alg))$. Hence we can define $\la_\alg\colon\cont_c(G,\M^\st(\alg))\to \Ls(L^2(G,\alg))$
by $\la_\alg(f)\xi\defeq f*\xi$ for all $f\in \cont_c(G,\M^\st(\alg))$ and $\xi\in \cont_c(G,\alg)$.
Here, $L^2(G,\alg)$ denotes the Hilbert $\alg$-module
defined as the completion of $C_c(G,\alg)$ with respect to the inner product
$\int_G\xi(t)^*\eta(t)\dd{t}$ for all $\xi,\eta\in C_c(G,\alg)$. Note that $\M(\alg\otimes \csg)$
can be seen in a canonical way as a \cstar{subalgebra} of $\Ls(L^2(G,\alg))$ and in this way
$\la_\alg(f)\in \MM{\alg\otimes\csg}$ for all $f\in \cont_c(G,\M^\st(\alg))$.
Now we can describe the result mentioned above (Proposition~4.6 in \cite{Buss:Fourier}):

\begin{proposition} \label{prop:positive-definite implica integravel}
Let $f\in \cont_c(G,\M^\st(\alg))$. The operator $\la_\alg(f)\in \Ls(L^2(G,\alg))$ is positive
if and only if the pointwise product $\Delta^{\frac{1}{2}}\cdot f$ is a positive-definite function. Moreover, in this case
the element $a\defeq \la_\alg(f)\in \MM{\alg\otimes\csg}$ is integrable \textup(with respect to the Plancherel weight on $\csg$\textup) and
we have $\hat{a}=f$. In particular, $(\id_\alg\otimes\f)(a)=\hat{a}(e)=f(e)$.
Moreover, the Fourier inversion Theorem give us the following formula for the operator $\la_\alg(f)$\textup:
$$\la_\alg(f)=\int_G^\su f(t)\otimes\la_t\dd{t}.$$
\end{proposition}

Note that the strict unconditional integral $\int^\su$ above coincides with the strict integral $\int^\st$
because the function $t\mapsto f(t)\otimes\la_t$ is strictly continuous and has compact support (and therefore is strictly Bochner integrable).
Therefore, we may rewrite the above formula as
\begin{equation}\label{eq:regular representation}
\la_\alg(f)=\int_G^\st f(t)\otimes\la_t\dd{t}.
\end{equation}

\subsection{Fell bundles}

In this section, we select some basic facts that will be used later. We start by recalling the definition of Fell bundles.
For more details on Fell bundles, we indicate \cite{Doran-Fell:Representations_2}.

\begin{definition} Let $G$ be a locally compact group. A \emph{Fell bundle} over $G$ is a (continuous)
Banach bundle $\fell=\{\fell_t\}_{t\in G}$ over $G$ endowed with a continuous multiplication
$$\fell\times\fell\to\fell,\quad (b,c)\mapsto b\cdot c$$ and a continuous involution
$$\fell\mapsto \fell,\quad b\mapsto b^*,$$ satisfying
\begin{enumerate}
\item[\textup{(i)}] $\fell_t\cdot\fell_s\sbe \fell_{ts}$ and $\fell_t^*=\fell_{t^{-1}}$ for all $t,s\in G$;
\item[\textup{(ii)}] $(a\cdot b)\cdot c=a\cdot(b\cdot c)$, $(a\cdot b)^*=b^*\cdot a^*$ and $(a^*)^*=a$ for all $a,b,c\in \fell$;
\item[\textup{(iii)}] $\|b\cdot c\|\leq\|b\|\|c\|$ and $\|b^*\|=\|b\|$ for all $b,c\in \fell$;
\item[\textup{(iv)}] $\|b^*b\|_{\fell_e}=\|b\|^2_{\fell_t}$ whenever $b\in \fell_t$ ($e$ denotes the identity of $G$); and
\item[\textup{(v)}] $b^*b\geq 0$ (in $\fell_e$) for all $b\in \fell$.
\end{enumerate}
\end{definition}
\pc
Let $\cont_c(\fell)$ be the space of compactly supported continuous sections of $\fell$.
For $\xi,\eta\in \cont_c(\fell)$, we define
$$(\xi*\eta)(t)\defeq \int_G\xi(s)\eta(s^{-1}t)\dd{s}$$
and
$$\xi^*(t)\defeq \fmod(t)^{-1}\xi(t^{-1})^*,$$
where $\fmod$ denotes the modular function of $G$.
Then $\xi*\eta\in \cont_c(\fell)$ and $\xi^*\in \cont_c(\fell)$, so that $\cont_c(\fell)$ becomes a \Star{}algebra.
By definition, $L^1(\fell)$ is the completion of $\cont_c(\fell)$ with respect to the $L^1$-norm:
$$\|\xi\|_1\defeq \int_G\|\xi(t)\|\dd{t},$$
and the \emph{full \cstar{algebra}} $\cstarfell$ of $\fell$ is the enveloping \cstar{algebra} of $L^1(\fell)$.

Let $L^2(\fell)$ be the Hilbert $\fell_e$-module  defined as the completion of $\cont_c(\fell)$
with respect to $\fell_e$-valued inner product:
$$\<\xi|\eta\>_{\fell_e}\defeq \int_\Gamma\xi(t)^*\eta(t)dx,$$
and the right $\fell_e$-action:
$$(\xi\cdot b)(t)\defeq \xi(t)\cdot b.$$
The \emph{left regular representation} of $\fell$ is the map
$$\la_\fell\colon\cstarfell\to \Ls\bigl(L^2(\fell)\bigr)$$
given by $\la_\fell(\xi)\eta=\xi*\eta$ for all $\xi,\eta\in \cont_c(\fell)$.
The \emph{reduced \cstar{algebra}} of $\fell$ is, by definition,
$\cstarfellr\defeq \la_\fell\bigl(\cstarfell\bigr)$.

\begin{example}\label{427} {\bf (1)} Let $(\algb,G,\beta)$ be a \cstar{dynamical system},
that is, a strongly continuous action $\beta\colon G\to\aut(\algb)$ of $G$ on a \cstar{algebra} $\algb$.
Then the trivial bundle $\fell\defeq \algb\times G$ with the algebraic operations
$$(b,t)\cdot(c,s)\defeq (b\beta_t(c),ts),\quad (b,t)^*\defeq (\beta_{t^{-1}}(b^*),t^{-1}),$$
is a Fell bundle over $G$, called the \emph{semidirect product} of $(\algb,G,\beta)$
and denoted by $\fell=\algb\times_\beta G$. In this case, the full \cstar{algebra}
$\cstarfell$ is isomorphic to the full crossed product $C^*(G,\algb)$, and the reduced \cstar{algebra} $\cstarfellr$
is isomorphic to the reduced crossed product $C^*_\red(G,\algb)$ of the dynamical system
$(\algb,G,\beta)$. In particular, if $\algb=\C$ with the trivial action of $G$, then $\cstarfell\cong C^*(G)$
and $\cstarfellr\cong C^*_\red(G)$. In this case, $\la_\fell$ corresponds to the left regular representation of $G$.

{\bf (2)} Let $(\algb,G,\beta)$ be a \cstar{dynamical system}. Suppose that $I$ is a (closed, two-sided) ideal of $\algb$.
For each $t\in G$, define $I_t\defeq  I\cap\beta_t(I)$ (which is also a closed, two-sided ideal of $\algb$) and the map
$\theta_t:I_{t^{-1}}\to I_t$ by $\theta_t(b)\defeq \beta_t(b)$. Then $\theta\defeq \{I_t,\theta_t\}_{t\in G}$ is a
\emph{partial action} of $G$ on $I$ as defined by Ruy Exel in \cite{Exel:TwistedPartialActions}.
It is called the \emph{restriction} of $\beta$ to the ideal $I$.

Define
$$\fell\defeq \{(b,t):b\in  I_t\}\sbe \algb\times G,\quad \fell_t\defeq I_t\times\{t\}\cong I_t$$
with the operations
$$(b,t)\cdot(c,s)\defeq (b\beta_t(c),ts),\quad (b,t)^*\defeq (\beta_{t^{-1}}(b^*),t^{-1}).$$
Then $\fell$ is a Fell bundle over $G$, called the \emph{semidirect product} of the partial dynamical system
$(I,G,\theta)$; it is denoted by $\fell=I\times_\theta G$. One has
$$\cstarfell\cong C^*(G,I,\theta)\quad\mbox{e}\quad \cstarfellr\cong C_\red^*(G,I,\theta),$$
where $C^*(G,I,\theta)$ and $C_\red^*(G,I,\theta)$ denote the full and reduced crossed products of $(I,G,\theta)$.

More generally, one can construct Fell bundles from
\emph{twisted partial actions}, and up to a certain regularity condition
all Fell bundles are of this form (see \cite{Exel:TwistedPartialActions} for details).
\end{example}

\begin{definition} \label{def:RepresentationFellBundle} Let $\hilm$ be a Hilbert $\algb$-module, and let $\fell=\{\fell_t\}$
be a Fell bundle over a locally compact group $G$.
A \emph{representation} of $\fell$ on $\hilm$ is a map
$\pi$ from $\fell$ to the \cstar{}algebra $\Ls(\hilm)$ of all adjointable operators on $\E$, satisfying
\begin{enumerate}
\item[\textup{(i)}] the restriction $\pi_{|_{\fell_t}}\colon\fell_t\to \Ls(\hilm)$ is linear for all $t\in G$;
\item[\textup{(ii)}] $\pi(bc)=\pi(b)\pi(c)$ for all $b,c\in \fell$;
\item[\textup{(iii)}] $\pi(b)^*=\pi(b^*)$ for all $b\in \fell$; and
\item[\textup{(iv)}] for each $\xi\in\hilm$, the map $\fell\ni b\to \pi(b)\xi\in \hilm$ is continuous.
\end{enumerate}
A representation $\pi$ of $\fell$ on $\hilm$ is called {\it nondegenerate} if $\cspn(\pi(\fell)\hilm)=\hilm$.
It is called {\it isometric} if $\|\pi(b)\|=\|b\|$ for all $b\in \fell$.
\end{definition}

Every representation $\pi$ of $\fell$ on $\hilm$ corresponds to a unique representation (still denoted by $\pi$) of $\cstarfell$,
called the \emph{integrated form} of $\pi$, which is determined by
$$\pi(f)\xi=\int_G\pi\bigl(f(t)\bigr)\xi \dd{t}\quad\mbox{ for all } f\in L^1(\fell)\mbox{ and } \xi\in \hilm.$$
This induces a bijective correspondence between representations of
$\fell$ and representations of $\cstarfell$. This correspondence
preserves nondegeneracy. Note that condition (iv) above implies that the function $G\ni t\mapsto \pi\bigl(f(t)\bigr)\in \Ls(\hilm)$ is
strongly continuous for all $f\in \cont_c(\fell)$. Since it is bounded, it is also strictly continuous. Since $\cont_c(\fell)$ is
dense in $L^1(\fell)$, it follows that the function $G\ni t\mapsto \pi\bigl(f(t)\bigr)\in \Ls(\hilm)=\MM{\K(\hilm)}$
is strictly mensurable for all $f\in L^1(\fell)$, that is,
for all $x\in \K(\hilm)$, a map $G\ni t\mapsto \pi\bigl(f(t)\bigr)x\in \K(\hilm)$ is measurable.
Here $\K(\E)$ denotes the \cstar{algebra} of compact operators on $\E$. Moreover, we have
$$\int_G\bigl\|\pi\bigl(f(t)\bigr)x\bigr\|\dd{t}\leq\left(\int_G\|f(t)\|\dd{t}\right)\|x\|<\infty.$$
It follows that the map $G\ni t\mapsto \pi\bigl(f(t)\bigr)\in \MM{\K(\hilm)}$ is strictly integrable for all $f\in L^1(\fell)$.
Hence the above formula for the integrated form can be rewritten as
\begin{equation}\label{383}
\pi(f)=\int_G^\st\pi\bigl(f(t)\bigr)\dd{t}\quad\mbox{ for all } f\in L^1(\fell),
\end{equation}
where the superscript "$\st$" indicates strict integration.

Let $\fell$ be a Fell bundle over $G$, and for each $t\in G$, let $\Phi_t\colon\fell_t\to \MM{\cstarfell}$ be the map defined by
$\Phi_t(b)\xi|_s=b\xi(t^{-1}s)$ and $\xi\Phi_t(b)|_s=\fmod(t^{-1})\xi(st^{-1})b$ for all $b\in \fell_t$,
$\xi\in \cont_c(\fell)$ and $s\in G$. Let $\Psi_t\colon\fell_t\to \MM{\cstarfellr}$ be the composition
$\Psi_t=\la_\fell\circ\Phi_t$. Note that, identifying
$\MM{\cstarfellr}\sbe \Ls\bigl(L^2(\fell)\bigr)$, $\Psi_t$ is given by
$\Psi_t(b)\xi|_s=b\xi(t^{-1}s)$ for all $\xi\in \cont_c(\fell)\sbe L^2(\fell)$.

\begin{proposition}\label{prop:embedding of Fell bundles} Let $\fell$ be a Fell bundle over $G$.
With the above notations, we define $\Phi\colon\fell\to \MM{\cstarfell}$ and $\Psi\colon\fell\to \MM{\cstarfellr}$ by
$\Phi(b)=\Phi_t(b)$ and $\Psi(b)=\Psi_t(b)$ for all $b\in \fell_t$.
Then $\Phi$ and $\Psi$ are nondegenerate, isometric representations
of $\fell$. The integrated forms are, respectively,
the inclusion $\cstarfell\hookrightarrow \MM{\cstarfell}$ and the
left regular representation $\la_\fell\colon\cstarfell\to \MM{\cstarfellr}$.
\end{proposition}
\begin{proof} It is easy to check conditions (i)--(iii) of Definition~\ref{def:RepresentationFellBundle}. To check
(iv), it is enough to show that the map $b\mapsto \Phi(b)\xi$ from $\fell$ into $\cont_c(\fell)$ is
continuous with respect to the inductive limit topology for all $\xi\in \cont_c(\fell)$. Suppose that
$(b_i)$ is a net in $\fell$ converging to some $b\in \fell$. Take $t_i,t\in G$ such that $b_i\in
\fell_{t_i}$ and $b\in \fell_t$. Note that $t_i\to t$ because the bundle projection $\fell\to G$ is continuous.
For each $i$, define the function $f_i\colon G\to \C$ by
$$f_i(s)\defeq \|(\Phi(b_i)\xi)(s)-(\Phi(b)\xi)(s)\|=\|b_i\xi(t_i^{-1}s)-b\xi(t^{-1}s)\|.$$
Note that $f_i$ belongs to $\cont_c(G)$. Since $t_i\to t$, we may assume that the net
$(t_i)$ is contained in a fixed compact subset of $G$. Hence the supports of the functions $f_i$ are all contained in a fixed compact
subset $K_0\sbe G$. Now for each $i$, there is $s_i\in K_0$ such that
$x_i\defeq \sup_{s\in G}f_i(s)=f_i(s_i)$. Passing to a subnet if necessary, we may assume that $(s_i)$ converges
to some $s\in K_0$. It follows that $x_i\to 0$ and hence $\Phi(b_i)\xi\to \Phi(b)\xi$ in the inductive limit topology.
Therefore, $\Phi$ and hence also $\Psi=\la_\fell\circ\Phi$ is a representation of $\fell$. The integrated form of $\Phi$ is given by
$$\Phi(f)\xi|_s=\int_G\Phi\bigl(f(t)\bigr)\xi(s)\dd{t}=\int_Gf(t)\xi(t^{-1}s)\dd{t}=f*\xi$$
for all $f,\xi\in \cont_c(\fell)$, that is, $\Phi$ is the canonical inclusion of $\cstarfell$ into $\MM{\cstarfell}$.
Thus the integrated form of $\Psi=\la_\fell\circ\Phi$ coincides with $\la_\fell$. In particular,
$\Phi$ and $\Psi$ are nondegenerate. Now suppose that
$b\in \fell_e$ and $\Phi(b)=0$. This means that $b\xi(t)=0$ for all
$\xi\in \cont_c(\fell)$ and $t\in G$. In particular, $b\xi(e)=0$ for all
$\xi\in \cont_c(\fell)$, which is equivalent to $bc=0$ for all
$c\in \fell_e$. Thus $b=0$ and hence $\Phi$, restricted to
$\fell_e$, is injective and therefore isometric. As a conclusion,
$$\|\Phi(b)\|^2=\|\Phi(b)^*\Phi(b)\|=\|\Phi(b^*b)\|=\|b^*b\|=\|b\|^2\quad\mbox{for all } b\in \fell$$
This shows that $\Phi$ is isometric. Similarly, one can show that $\Psi$ is isometric.
\end{proof}

\begin{remark} Let $\pi\colon\fell\to \Ls(\hilm)$ be a nondegenerate representation of $\fell$, and let $\tilde\pi$ be its integrated form.
Then $\tilde\pi\circ\Phi=\pi$. In fact, for all $b\in \fell_t$, $f\in \cont_c(\fell)$ and $\xi\in \hilm$, we have
\begin{multline*}
\tilde\pi\bigl(\Phi(b)\bigr)\tilde\pi(f)\xi=\tilde\pi\bigl(\Phi(b)f\bigr)\xi=\int_G\pi\bigl((\Phi(b)f)(s)\bigr)\xi \dd{s}\\
=\int_G\pi\bigl(b_tf(t^{-1}s)\bigr)\xi \dd{s}=\pi(b_t)\int_G\pi\bigl(f(s)\bigr)\xi \dd{s}=\pi(b)\tilde\pi(f)\xi.
\end{multline*}
In particular, if $\tilde\pi$ is isometric, so is $\pi$. Since the regular representation
$\la_\fell$ is not injective in general, Proposition~\ref{prop:embedding of Fell bundles}
shows that the converse is not true.
\end{remark}

\section{Integrable group coactions}

\subsection{Preliminaries on group coactions}

Firstly, let us recall what is a group coaction. Let $G$ be a locally compact group and consider on the
reduced \cstar{algebra} $\csg$ of $G$, the usual \emph{comultiplication} $\com$, that is, the nondegenerate \Star{}homomorphism
$$\com\colon\csg\to\M\bigl(\csg\otimes\csg\bigr)$$ determined by the equation
$\com(\la_t)=\la_t\otimes\la_t$ for all $t\in G$. The fact that $\com$ is a comultiplication
is expressed by the identity $(\com\otimes\id)\circ\com=(\id\otimes\com)\circ\com$.

\begin{definition} Let $\alg$ be a \cstar{algebra}. A \emph{coaction} of $G$ on $\alg$ is a nondegenerate \Star{}homomorphism
$\co{\alg}\colon\alg\to\MM{\alg\otimes\csg}$ satisfying the identity $(\co{\alg}\otimes\id)\circ\co{\alg}=(\id\otimes\com)\circ\co{\alg}$,
and such that $\co{\alg}(\alg)\cdot(1\otimes\csg)$ spans a dense subspace in $\alg\otimes\csg$. A \cstar{algebra} with a coaction of $G$ is
also called a \emph{\dualgcstar{algebra}}.
\end{definition}

Note that the linear span of $\co{\alg}(\alg)\cdot(1\otimes\csg)$ is dense in $\alg\otimes\csg$ if and only the same holds for
$(1\otimes\csg)\cdot\co{\alg}(\alg)$. Hence a coaction $\co{\alg}$ has its image contained in the \cstar{subalgebra}
$\tildeMM{\alg\otimes\csg}\sbe\MM{\alg\otimes\csg}$ defined by
\begin{multline*}
\left\{x\in \MM{\alg\otimes\csg}\colon x\cdot(1\otimes\csg), (1\otimes\csg)\cdot x\sbe \alg\otimes\csg\right\}.
\end{multline*}

Given a coaction $\co{\alg}$ of $G$ on a \cstar{algebra} $\alg$, we can define a (left) action of
the (reduced) Fourier-Stieltjes algebra $\fourst{G}$ of the group $G$ on $\alg$ by:
$$\omega*\xi\defeq (\id\otimes\omega)\bigl(\co{\alg}(\xi)\bigr)\quad\mbox{for all }\omega\in\fourst{G}, \xi\in \alg.$$
Recall that $\fourst{G}$ consists of bounded continuous functions $G\to\C$ of the form $t\mapsto\omega(\la_t)$, where
$\omega\in \csg^*$ is some bounded linear functional on $\csg$. The multiplication on $\fourst{G}$ is, by definition,
the pointwise product of functions. Alternatively, $\fourst{G}$ is isomorphic to the dual $\csg^*$ endowed with the multiplication
$\omega\cdot\theta\defeq (\omega\otimes\theta)\circ\com$ for $\omega,\theta\in \csg^*$. Note that this isomorphism was implicitly used
in the definition of the operation $\omega*\xi$ above. The isomorphism $\fourst{G}\cong\csg^*$
is given in such way that the function $t\mapsto \omega(\la_t)$ in $\fourst{G}$ corresponds to the functional $\omega$ in $\csg^*$.
In particular, $\fourst{G}$ is a Banach algebra. Another important algebra in this context is the Fourier algebra $\fourier{G}$
of $G$ which is the (closed) subalgebra of $\fourst{G}$ generated by the functions $t\mapsto \braket{u}{\la_t(v)}$,
with $u,v\in L^2(G)$ (see \cite{Eymard:Fourier}). In particular, we can restrict the action of
$\fourst{G}$ on $\alg$ defined above to the Fourier subalgebra $\fourier{G}$.
In fact, this action is always nondegenerate for every \dualgcstar{algebra} $\alg$, that is,
the linear span of all elements $\omega*\xi$ with
$\omega\in \fourier{G}$ and $\xi\in \alg$ is dense in $\alg$.

The class of all \dualgcstar{algebras} forms a category. The morphisms are the \Star{}homomorphisms $\pi\colon\alg\to\algb$ which are
\emph{$\dualg$-equivariant}, meaning that $(\pi\otimes\id)\co{\alg}=\co{\algb}\circ\pi$ or, equivalently, $\pi$ respects the
induced actions of $\fourst{G}$, that is, $\pi(\omega*\xi)=\omega*\pi$ for all $\omega\in \fourst{G}$ and $\xi\in \alg$.
In fact, it is enough to require this for $\omega$ in the Fourier subalgebra $\fourier{G}\sbe\fourst{G}$.

\begin{example} \label{exa:Coactions}
{\bf (1)} Any \cstar{algebra} $\alg$ can be equipped with the trivial coaction $\co{\triv}$ defined by
$\co{\triv}(a)\defeq a\otimes 1$ for all $a\in \alg$.

{\bf (2)} The comultiplication $\com\colon\csg\to\MM{\csg\otimes\csg}$ can be seen as a coaction of $G$ on $\csg$.
One can also consider the canonical coaction $\co{G}\colon\csgf\to\MM{\csgf\otimes\csg}$ of $G$ on full \cstar{algebra} $\csgf$
which is determined by the identity $\co{G}(\repu_t)=\repu_t\otimes\la_t$ for all $t\in G$,
where $\repu\colon G\to\MM{\csgf}$ denotes the \emph{universal representation} of $G$ defined
by $\repu_t(f)|_s\defeq f(t^{-1}s)$ for $t,s\in G$ and $f\in C_c(G)\sbe\csgf$.

{\bf (3)} Assume that $G$ is abelian. Then coactions of $G$ correspond bijectively to (strongly continuous) actions of the
Pontrjagin dual $\dualg$ through the canonical identification of $\tildeMM{\alg\otimes\csg}$ with the \cstar{algebra}
$\cont_b(\dualg,\alg)$ of bounded continuous functions of $\dualg$ in $ \alg$.
This example explains the terminology "\dualgcstar{algebras}".

{\bf (4)} A class of important examples are the \emph{dual coactions}. Given a (strongly continuous) action
of $G$ on a \cstar{algebra} $\algb$, there are canonical coactions, both called \emph{dual coactions},
on the full and reduced crossed products $\crossp{G}{\algb}$ and $\crosspr{G}{\algb}$. Roughly speaking, these coactions
act trivially on $\algb$ and through the canonical coactions defined in (2) on $\csgf$ and $\csg$. More generally, we may consider
a Fell bundle $\fell=\{\fell_t\}_{t\in G}$ over $G$, and on its \cstar{}algebras $\cstarfell$ and $\cstarfellr$ there are
natural coactions $\co{\fell}$ and $\co{\fell}^\red$, respectively, also called \emph{dual coactions}.
The coaction $\co{\fell}$ is determined by the equation $\co{\fell}(b_t)=b_t\otimes\la_t$ whenever $b_t\in \fell_t$.
Here we identify the fibers $\fell_t$ as subspaces of $\MM{\cstarfell}$ as in Proposition~\ref{prop:embedding of Fell bundles}.
And the coaction $\co{\fell}^\red$ is determined by the fact that the regular representation $\la_{\fell}\colon\cstarfell\to \cstarfellr$
is $\dualg$-equivariant. See \cite{ExelNg:Approx} for more details on the definition of dual coactions.
We shall describe and study these coactions more precisely in the next sections.
\end{example}

\subsection{Integrable coactions}

Let $G$ be a locally compact group. The pair $(\csg,\com)$ is a \emph{locally compact quantum group} in
the sense of  Kustermans and Vaes \cite{Kustermans-Vaes:LCQG}. The \emph{Haar weight} -- the noncommutative analogue of the Haar measure --
on $(\csg,\com)$ is the Plancherel weight defined in Section~\ref{533}. The Haar weight in this case is both right and left invariant,
that is, $(\csg,\com)$ is \emph{unimodular} as a quantum group.

Coactions of the group $G$ are, by definition, coactions of the quantum group $(\csg,\com)$.
The notion of integrable coactions is in this way a special case of the same concept for locally compact quantum groups
which we developed recently in \cite{Buss-Meyer:Square-integrable}. For reader's convenience, we recall the definition and
some basic facts about integrable coactions.

\begin{definition}
Let $(\alg,\co{\alg})$ be a \dualgcstar{algebra}. A positive element $a\in \M(\alg)^+$ is called \emph{integrable}
(with respect to the coaction $\co{\alg}$) if $\co{\alg}(a)$ is an integrable element in $\M\bigl(\alg\otimes\csg\bigr)$
with respect to the Plancherel weight $\f$ on $\csg$.
In other words, $a$ is integrable if $\co{\alg}(a)$ belongs to $\MM{\alg\otimes\csg}_\ii$.

The set of positive integrable elements in $\M(\alg)^+$ will be denoted by $\M(\alg)^+_\ii$ and the subset of integrable
elements in $\alg^+\sbe\M(\alg)^+$ will be denoted by $\alg_\ii^+$. In other words, $\alg_\ii^+=\M(\alg)_\ii^+\cap \alg$.
We write $\M(\alg)_\ii\defeq \spn{\M(\alg)^+_\ii}$ and $\alg_\ii\defeq \spn{\alg_\ii^+}$. Elements
in $\M(\alg)_\ii$ or in $\alg_\ii$ are also called \emph{integrable}.
We also write $\M(\alg)_\si\defeq \{a\in \M(\alg)\colon a^*a\in \M(\alg)_\ii^+\}$
for the subset of \emph{square-integrable} elements in $\M(\alg)$.
The subset of square-integrable elements in $\alg$ will be denoted by $\alg_\si=\M(\alg)_\si\cap \alg$.

The set of all integrable elements $\M(\alg)_\ii$ is a hereditary \Star{}subalgebra of $\M(\alg)$ whose positive part coincides
with $\M(\alg)_\ii^+$. Moreover, $\M(\alg)_\si$ is a left ideal in $\M(\alg)$ and
$$\M(\alg)_\ii=\spn\M(\alg)_\si^*\M(\alg)_\si.$$
Similar properties hold for $\alg_\ii$ and $\alg_\si$. In general, the subspaces $\M(\alg)_\ii$ and $\M(\alg)_\si$
may be "small". For instance, they may be zero. Integrability means that these subspaces are dense. More precisely,
the \dualgcstar{algebra} $(\alg,\co{\alg})$ is called \emph{integrable} if $\alg_\ii$ is dense in $\alg$ or, equivalently,
if $\alg_\ii^+$ is dense in $\alg^+$, or also if $\alg_\si$ is dense in $\alg$. It is also equivalent to
require that the subsets $\M(\alg)_\ii^+\sbe\M(\alg)^+$, $\M(\alg)_\ii\sbe\M(\alg)$ or $\M(\alg)_\si\sbe\M(\alg)$
are dense in $\M(\alg)$ with respect to the strict topology (see \cite[Lemma 4.12]{Buss-Meyer:Square-integrable} for more details).
\end{definition}

The map $\Av\colon\M(\alg)_\ii\to\M(\alg)$ defined by $\Av(a)\defeq (\id_\alg\otimes\f)(\co{\alg}(a))$ is called the \emph{averaging map}.
Lemma 4.10 in \cite{Buss-Meyer:Square-integrable} shows that the image of $\Av$ is contained in the fixed point algebra:

$$\M_1(\alg)\defeq \{x\in \M(\alg)\colon\co{\alg}(x)=x\otimes 1\}.$$

\begin{example}\label{exa:IntegrableCoactions}
{\bf (1)} If $G$ is discrete (or, equivalent, if the quantum group $(\csg,\com)$ is compact), then
the Plancherel weight is bounded, and hence every coaction is integrable. Otherwise, there are always non-integrable coactions.
For example, the trivial coaction on a non-zero \cstar{algebra} is integrable if and only if $G$ is discrete.
The same is true for coactions on unital \cstar{algebras}
because in this case there is no dense (left or right) ideal.

{\bf (2)} If $G$ is abelian, then integrable coactions of $G$ correspond to integrable actions
-- as defined by Rieffel and Exel in \cite{Exel:Unconditional,Exel:Spectral,Rieffel:Integrable_proper} -- of the dual group $\dualg$
through the usual correspondence between coactions of $G$ and actions of $\dualg$ as in Example~\ref{exa:Coactions}(3).
Moreover, if a \cstar{algebra} $\alg$ has a coaction of $G$ that corresponds to an action $\alpha$ of $\dualg$,
then an element $a\in \M(\alg)^+$ is integrable if and only if the strict unconditional integral
$\int_{\dualg}^\su\alpha_x(a)\dd{x}$ converges. In this case, this integral is equal to $\Av(a)$.
See comments after Definition~\ref{367}.

{\bf (3)} One of the most simple non-trivial examples of integrable coactions is the coaction $\com$ of $G$ on $\csg$, that is,
the comultiplication of $\csg$ itself. The canonical coaction $\co{G}$ of $G$ on its full \cstar{algebra}
$\csgf$ (see Example~\ref{exa:Coactions}(2)) is also integrable. More generally, every dual coaction
(see Example~\ref{exa:Coactions}(4)) on the full $\crossp{G}{\algb}$ or reduced $\crosspr{G}{\algb}$ crossed product is integrable,
for every \cstar{dynamical system} $(\algb,\beta,G)$. In fact, it is possible to generalize this result to quantum groups, see
\cite[Proposition~4.20]{Buss-Meyer:Square-integrable}.

{\bf (4)} As a special case of (3), we get that the dual coaction on $\K(L^2(G))\cong\crossp{G}{\cont_0(G)}\cong\crosspr{G}{\cont_0(G)}$
is integrable, where $\cont_0(G)$ is endowed with the translation action of $G$.
More generally, given any \dualgcstar{algebra} $\alg$, there is a canonical coaction on $\alg\otimes\K(L^2(G))$,
called the \emph{stable coaction}, which also integrable. This result can be also generalized to quantum groups,
see \cite[Proposition 4.22]{Buss-Meyer:Square-integrable}. Hence, up to stabilization, every coaction is integrable.
In particular, every coaction is Morita equivalent to an integrable coaction.
\end{example}

To prove that the comultiplication $\com$ on $\csg$ is integrable as a coaction of $G$, one simply uses that the
Plancherel weight is left invariant, see \cite[Example 4.4]{Buss-Meyer:Square-integrable}. And a similar argument also shows
that the coaction $\co{G}$ of $G$ on $\csgf$ is integrable. These facts together with the following permanence property
for integrable coactions \cite[Proposition 4.14]{Buss-Meyer:Square-integrable} show that every classical dual coaction is integrable:

\begin{proposition}\label{prop:PermanencePropertiesIntegrableCoactions}
Suppose that $\alg$ and $\algb$ are \dualgcstar{algebras} and that $\pi\colon\alg\to\M(\algb)$
is a nondegenerate $\dualg$-equivariant \Star{}homomorphism. Then $\algb$ is integrable whenever $\alg$ is.
Moreover, if $a\in\M(\alg)_\ii$, then $\pi(a)\in \M(\algb)_\ii$ and $\Av(\pi(a))=\Av(a)$.
\end{proposition}

One can replace the homomorphism $\pi$ in the above proposition by any
linear positive, nondegenerate, strictly continuous and $\dualg$-equivariant map $\M(\alg)\to\M(\algb)$.

Using the above proposition, it is simple to show that dual coactions on crossed products are integrable. In fact,
given any action $(\algb,\beta,G)$, there are canonical equivariant homomorphisms
$\csgf\to\MM{\crossp{G}{\algb}}$ and $\csg\to\MM{\crosspr{G}{\algb}}$.
This argument does not apply to dual coactions on the \cstar{algebras} $C^*_{(\red)}(\fell)$
of a Fell bundle $\fell=\{\fell_t\}_{t\in G}$ because there is no nondegenerate $\dualg$-equivariant homomorphisms
$C^*_{(\red)}(G)\to \MM{C^*_{(\red)}(\fell)}$. In fact, Lansdstad's Duality Theorem \cite{KaliszewskiQuigg:Landstad,Landstad:Duality}
shows that the existence of such homomorphisms is a peculiarity of classical dual coactions on crossed products.

One of the main goals in this work is to show that dual coactions on \dualgcstar{}algebras of Fell bundles are integrable.
We shall prove this using the definition of integrability rather of any other property.

\subsection{Integrability of dual coactions}
This section contains the main result of this work, namely, the fact that
dual coactions on full and reduced \cstaralgs{} of a Fell bundle are integrable.

Let $G$ be a locally compact group, let $\fell=\{\fell_t\}_{t\in G}$ be a Fell bundle over $G$ and
let $\alg\defeq \cstarfell$ be the full \cstar{algebra} of $\fell$. We identify
each $b_s\in \fell_s$ as an element of $\M(\alg)$ through the map $\Phi$ defined in Proposition~\ref{prop:embedding of Fell bundles}.
Thus $b_s\in \fell_s$ will be identified with the multiplier of $\alg=\cstarfell$ given by
$(b_s\cdot \xi)(t)\defeq b_s\cdot \xi(s^{-1}t)$ for all $\xi\in \cont_c(\fell)$ and $s,t\in G$.
With this identification, each section $\xi\in \cont_c(\fell)$ may be viewed as an element of $\cont_c\bigl(G,\M^{\st}(\alg)\bigr)$,
the space of compactly supported strictly continuous functions $G\to\M(\alg)$.
Since the integrated form of $\Phi\colon\fell\to \M(\alg)$ coincides with the inclusion
$\alg\hookrightarrow \M(\alg)$, Equation~\eqref{383} yields
\begin{equation}\label{eq:SectionAsIntegral}
\xi=\int_G^\st\!\!\Phi\bigl(\xi(s)\bigr)\,\dd{s}=\int_G^\st\!\!\xi(s)\,\dd{s},\quad \xi\in \cont_c(\fell),
\end{equation}
where the superscript "s" indicates a strict integral.
The same equation holds for $\xi\in L^1(\fell)$, but we only need it for $\xi\in \cont_c(\fell)$.

There is a \emph{dual coaction} of $G$ on $\alg$
denoted by $\co{\fell}\colon\alg\to \M(\alg\otimes \csg)$ and characterized by
$\co{\fell}(b_s)= b_s\otimes\la_s$ for all $b_s\in \fell_s$ (see \cite{ExelNg:Approx} for details).
It follows from Equation~\eqref{eq:SectionAsIntegral} that
\begin{equation}\label{139}
\co{\fell}(\xi)=\int_G^\st\!\!\xi(s)\otimes\la_s \dd{s},\quad \xi\in \cont_c(\fell)
\end{equation}
Comparing the above formula with Equation~\eqref{eq:regular representation}, we get the following formula for the dual coaction:
\begin{equation}\label{eq:FormulaDualCoaction}
\co{\fell}(\xi)=\la_\alg(\xi)\quad\mbox{for all }\xi\in \cont_c(\fell).
\end{equation}

As an immediate application of the formula above, we describe the $\fourst{G}$-action induced by the dual coaction.
Recall that $\fourst{G}$ denotes the Fourier-Stietjes algebra of $G$.
Given $u\in \fourst{G}\cong \csg^*$ and $\xi\in \cont_c(\fell)$, Equation~\eqref{eq:SectionAsIntegral} yields
\begin{equation}\label{223}
u*\xi=(\id\otimes u)\bigl(\co{\fell}(\xi)\bigr)
=(\id\otimes u)\left(\int_G^\st\xi(s)\otimes\la_s \dd{s}\right)=\int_G^\st\xi(s)u(s)\dd{s}=u\cdot\xi.
\end{equation}
Thus, the action of $\fourst{G}$ on $\alg$ induced by the dual coaction $\co{\fell}$,
when restricted to $\cont_c(\fell)$, is simply pointwise multiplication.

As a second, more important application of Equation~\eqref{eq:FormulaDualCoaction},
we show that the dual coaction $\co{\fell}$ is integrable:

\begin{theorem}\label{the:DualCoactionIntegrable} Let $\fell$ be a Fell bundle over $G$.
Then the \cstar{algebra} $\cstarfell$ endowed with the
dual coaction $\co{\fell}$ is an integrable \dualgcstar{algebra}.
Moreover, each element in $\cont_c(\fell)$ is square-integrable, that is,
each element $\xi$ of the dense $*$-subalgebra
$\cont_c(\fell)^2\defeq \spn\bigl\{\eta*\zeta\colon\eta,\zeta\in \cont_c(\fell)\bigr\}$ is integrable, and
we have
\begin{equation}\label{eq:AverageMapDualCoaction}
\Av(\xi)=\xi(e).
\end{equation}
\end{theorem}
\begin{proof} Let $\alg\defeq \cstarfell$. By polarization, it is enough to show that
$\xi=\eta^**\eta$ is integrable for all $\eta\in \cont_c(\fell)$.
Identifying $\cont_c(\fell)\sbe \cont_c\bigl(G,\M^\st(\alg)\bigr)$ as above,
we have $\co{\fell}(\xi)=\la_\alg(\xi)$ by Equation~\eqref{eq:FormulaDualCoaction}.
Since $\la_\alg(\xi)=\la_\alg(\eta)^*\la_\alg(\eta)\geq 0$, Proposition~\ref{prop:positive-definite implica integravel}
implies that $\co{\fell}(\xi)$ belongs to $\MM{\alg\otimes\csg}_\ii$, that is, $\xi\in \alg_\ii$, and
\begin{equation*}
\Av(\xi)=(\id_\alg\otimes\f)\bigl(\co{\fell}(\xi)\bigr)=(\id_\alg\otimes\f)\bigl(\la_\alg(\xi)\bigr)=\xi(e).
\end{equation*}
\vskip-14pt
\end{proof}

Now we direct our attention to the reduced \cstar{algebra} $\cstarfellr$ of a Fell bundle $\fell$.
Let us start recalling that the left regular representation of $\fell$ is defined by
$$\la_\fell\colon\fell\to\Ls\bigl(L^2(\fell)\bigr),\quad \la_\fell(b_t)\xi(s)\defeq b_t\xi(t^{-1}s)\quad\mbox{for all }
b_t\in \fell_t\mbox{ and }\,s,t\in G,$$
where $L^2(\fell)$ denotes the Hilbert $\fell_e$-module defined as the completion of $\cont_c(\fell)$
with the obvious right $\fell_e$-action and the $\fell_e$-inner product
$$\<\xi,\eta\>_{\fell_e}\defeq \int_G\xi(t)^*\eta(t)\dd{t},\quad \xi,\eta\in \cont_c(\fell).$$
By definition,
$$\cstarfellr\defeq \la_\fell\bigl(\cstarfell\bigr)\sbe \Ls\bigl(L^2(\fell)\bigr).$$
There is a coaction $\co{\fell}^\red$ of $G$ on $\cstarfellr$, also called \emph{dual coaction},
determined by the fact that $\la_\fell\colon\cstarfell\to \cstarfellr$ is equivariant
(see \cite{ExelNg:Approx} for details). Moreover, $\co{\fell}^\red$ satisfies
$$\co{\fell}^\red(x)=W_\fell(x\otimes 1)  W_\fell^*,$$
where $W_\fell$ is the unitary operator in $\Ls\left(L^2(\fell)\otimes L^2(G)\right)$ defined by $W_\fell\xi(s,t)=\xi(s,s^{-1}t)$ for all
$\xi\in \cont_c(\fell\times G)$ and $s,t\in G$. Here we identify $L^2(\fell)\otimes L^2(G)\cong L^2(\fell\times G)$, where
$\fell\times G$ denotes the pullback of $\fell$ along the projection $G\times G\ni (s,t)\mapsto s\in G$. In particular,
this implies that $\co{\fell}^\red$ is injective.

The equivariance of $\la_\fell$ yields
$$\co{\fell}^\red\bigl(\la_\fell(b_t)\bigr)=(\la_\fell\otimes\id)\bigl(\co{\fell}(b_t)\bigr)=(\la_\fell\otimes\id)(b_t\otimes\la_t)=
\la_\fell(b_t)\otimes\la_t\quad\mbox{for all } b_t\in \fell_t,$$
and hence
$$\co{\fell}^\red\bigl(\la_\fell(\xi)\bigr)=\int_G^\st\la_\fell\bigl(\xi(t)\bigr)\otimes \la_t\dd{t},\quad \xi\in \cont_c(\fell).$$

Using the regular representation $\la_\fell\colon\cstarfell\to\cstarfellr$,
we can carry over our results from $\cstarfell$ to $\cstarfellr$:

\begin{corollary}\label{cor:RedDualCoactionIntegrable}
Let $\fell$ be a Fell bundle over $G$. Then the \cstar{algebra} $\cstarfellr$ endowed with the dual coaction
$\co{\fell}^\red$ of $G$ is an integrable \dualgcstar{algebra}. Moreover, every element in
$\la_\fell\bigl(\cont_c(\fell)\bigr)$ is square-integrable, that is,
every element $\la_\fell(\xi)$ in $\la_\fell\bigl(\cont_c(\fell)^2\bigr)$ is integrable and
$$\Av\bigl(\la_\fell(\xi)\bigr)=\la_\fell\bigl(\xi(e)\bigr).$$
\end{corollary}
\begin{proof} The result follows from Theorem~\ref{the:DualCoactionIntegrable} and Proposition~\ref{prop:PermanencePropertiesIntegrableCoactions}.
\end{proof}

\subsection{The Fourier transform}

\begin{definition}\label{FourierCoef} Let $(\alg,\co{\alg})$ be a \dualgcstar{algebra}. For $a\in \M(\alg)_\ii$,
we define the \emph{Fourier coefficient} of $a$ at $t\in G$ by
$$E_t(a)\defeq (\id_\alg\otimes\f)\bigl((1\otimes\la_t^{-1})\co{\alg}(a)\bigr).$$
The map $t\mapsto E_t(a)$ from $G$ into $\M(\alg)$ is called the \emph{Fourier transform} of $a$.
\end{definition}

In other words, the Fourier transform of an integrable element $a\in\M(\alg)_\ii$
is the Fourier transform of the integrable element $\co{\alg}(a)\in \MM{\alg\otimes\csg}_\ii$ as in Definition~\ref{FourierCoefficient},
that is, $E_t(a)=\widehat{\co{\alg}(a)}(t)$ for all $t\in G$. Note also that
$$E_e(a)=(\id_{\alg}\otimes\f)(\co{\alg}(a))=\Av(a).$$

If the group $G$ is abelian and if we identify $\csg\cong \cont_0(\dualg)$ in the usual way (through the Fourier transform), then
$(1_\alg\otimes\la_{t^{-1}})(f)(x)=\braket{x}{t}f(x)$ for all $f\in \cont_b\bigl(\dualg,\M^{\st}(\alg)\bigr)\cong \M(\alg\otimes \csg)$
and $x\in \dualg$, where we write $\braket{x}{t}\defeq x(t)$ to emphasize the duality between $G$ and $\dualg$.
Thus if $\co{\alg}$ corresponds to an action $\alpha$ of $\dualg$ on $\alg$,
then the Fourier coefficient coincides with that one defined by Exel in \cite{Exel:Unconditional,Exel:Spectral}:
$$E_t(a)=\int_{\dualg}^\su\langle x,t\rangle\alpha_x(a)\dd{x},\quad a\in \M(\alg)_\ii.$$

\begin{proposition}\label{prop:Fourier coef in espaco spectral} Let $\alg$ be a \dualgcstar{algebra}
and let $a\in \M(\alg)_\ii$ be an integrable element.
Then the Fourier coefficient $E_t(a)$ belongs to the $t$-spectral subspace
$\M_t(\alg)$ of $\M(\alg)$ defined by
$$\M_t(\alg)\defeq \{b\in \M(\alg)\colon\co{\alg}(b)=b\otimes \la_t\}.$$
\end{proposition}
\begin{proof} Since the comultiplication $\com$ of $\csg$ satisfies $\com(\la_t)=\la_t\otimes\la_t$, the equation
$(\co{\alg}\otimes\id)\circ\co{\alg}=(\id_\alg\otimes\com)\circ\co{\alg}$ yields
$$(\co{\alg}\otimes\id)\bigl((1_\alg\otimes \la_t^{-1})\co{\alg}(a)\bigr)
		=(1_\alg\otimes \la_t\otimes 1)(\id_\alg\otimes\com)\bigl((1_\alg\otimes\la_t^{-1})\co{\alg}(a)\bigr).$$
Now, using the invariance of the weight $\f$ with respect to the comultiplication $\com$
in its generalized form as in \cite[Proposition 3.1]{Kustermans-Vaes:LCQGvN},
we get the desired result:
\begin{align*}
\co{\alg}\bigl(E_t(a)\bigr)&
			  =\co{\alg}\Bigl((\id_\alg\otimes\f)\bigl((1_\alg\otimes\la_t^{-1})\co{\alg}(a)\bigl)\Bigr)
	\\&=(\id_\alg\otimes\id\otimes\f)\Bigl((\co{\alg}\otimes\id)\bigl((1_\alg\otimes \la_t^{-1})\co{\alg}(a)\bigr)\Bigr)
	\\&=(\id_\alg\otimes\id\otimes\f)\Bigl((1_\alg\otimes \la_t\otimes 1)(\id_\alg\otimes\com)
                                                                                    \bigl((1_\alg\otimes\la_t^{-1})\co{\alg}(a)\bigr)\Bigr)
	\\&=(1_\alg\otimes \la_t)(\id_\alg\otimes\id\otimes\f)\Bigl((\id_\alg\otimes\com)\bigl((1_\alg\otimes\la_t^{-1})\co{\alg}(a)\bigr)\Bigr)
	\\&=(1_\alg\otimes \la_t)\Bigl((\id_\alg\otimes\f)\bigl((1_\alg\otimes\la_t^{-1})\co{\alg}(a)\bigr)\otimes 1\Bigr)
	\\&=E_t(a)\otimes\la_t.
\end{align*}
\vskip-14pt
\end{proof}

\par Let $e\in G$ be the identity element. Note that the $e$-spectral subspace $\M_e(\alg)$ is exactly the fixed point algebra:
$$\M_e(\alg)=\M_1(\alg)=\{b\in \M(\alg)\colon\co{\alg}(b)=b\otimes 1\}.$$
The following result is a generalization of \cite[Proposition~6.4]{Exel:Spectral}.
\begin{proposition}\label{020} Let $\alg$ be a \dualgcstar{algebra} and consider $a,b\in \M(\alg)_\ii$ and $m\in \M_s(\alg)$.
Then, for all $s,t\in G$,
\begin{enumerate}
\item[\textup{(i)}] $E_t(a)^*=\fmod(t)^{-1}E_{t^{-1}}(a^*)$,
\item[\textup{(ii)}] $ma\in \M(\alg)_\ii$ and $mE_t(a)=E_{st}(ma)$,
\item[\textup{(iii)}] $am\in \M(\alg)_\ii$ and $E_t(a)m=\fmod(s)E_{ts}(am)$,
\item[\textup{(iv)}] $E_t(a)E_s(b)=E_{ts}\bigl(E_t(a)b\bigr)=\fmod(s)E_{ts}\bigl(aE_s(b)\bigl)$.
\end{enumerate}
\end{proposition}
\begin{proof} \begin{enumerate}
\item[\textup{(i)}] Recall that $\la_t$ is an analytic element and
$\sigma_z(\la_t)=\fmod(t)^{\ii z}\la_t$ for all $z\in \C$ (see Equation~\eqref{564}).
This implies that $\f(x\la_t)=\f(\sigma_\ii(\la_t)x)=\fmod(t)^{-1}\f(\la_tx)$ whenever $x\in\MM{\csg}_\ii$;
see Proposition~1.12 in \cite{Kustermans-Vaes:Weight}.
A generalization of this fact also holds for the slice map $\id_\alg\otimes\f$; see Proposition 3.28 in \cite{Kustermans-Vaes:Weight}. Thus
\begin{align*}
E_t(a)^*&=(\id_\alg\otimes\f)\bigl((1_\alg\otimes\la_t^{-1})\co{\alg}(a)\bigr)^*
		\\&=(\id_\alg\otimes\f)\bigl(\co{\alg}(a^*)(1_\alg\otimes\la_t)\bigr)
		\\&=(\id_\alg\otimes\f)\bigl((1_\alg\otimes\sigma_i(\la_t))\co{\alg}(a^*)\bigr)
		\\&=\fmod(t)^{-1}E_{t^{-1}}(a^*).
\end{align*}
\item[\textup{(ii)}]

If $b\in \M(\alg)_\si$, then $\co{\alg}(bm)=\co{\alg}(b)(m\otimes\la_t)\in \M(\alg\otimes \csg)_\si$ because $\la_t$
is an analytic element. In other words, $bm\in \M(\alg)_\si$.
Using that $\M(\alg)_\ii$ is linearly spanned by $\M(\alg)_\si^* \M(\alg)_\si$, this implies that $ma\in \M(\alg)_\ii$.
Moreover, since $\co{\alg}(m)=m\otimes\la_s$, we conclude that
\begin{align*}
mE_t(a)&=m(\id_\alg\otimes\f)\bigl((1_\alg\otimes\la_{t^{-1}})\co{\alg}(a)\bigr)
		\\&=(\id_\alg\otimes\f)\bigl((m\otimes\la_{t^{-1}})\co{\alg}(a)\bigr)
		 \\&=(\id_\alg\otimes\f)\bigl((1_\alg\otimes\la_{t^{-1}}\la_{s^{-1}})\co{\alg}(ma)\bigr)
		\\&=E_{st}(ma).
\end{align*}
\item[\textup{(iii)}] As in (ii) one can prove that $am\in \M(\alg)_\ii$.
Using again the analyticity of $\la_s$ and the relations $\sigma_z(\la_s)=\fmod(s)^{\ii z}\la_s$ and $\co{\alg}(m)=m\otimes\la_s$, we get
\begin{align*}
E_t(a)m&=(\id_\alg\otimes\f)\bigl((1_\alg\otimes\la_{t^{-1}})\co{\alg}(a)\bigr)m
		\\&=(\id_\alg\otimes\f)\bigl((1_\alg\otimes\la_{t^{-1}})\co{\alg}(a)(m\otimes 1)\bigr)
		 \\&=(\id_\alg\otimes\f)\bigl((1_\alg\otimes\la_{t^{-1}})\co{\alg}(am)(1_\alg\otimes\la_{s^{-1}})\bigr)
		 \\&=(\id_\alg\otimes\f)\bigl((1_\alg\otimes\sigma_i(\la_{s^{-1}})\la_{t^{-1}})\co{\alg}(am)\bigr)
		 \\&=\fmod(s)(\id_\alg\otimes\f)\bigl((1_\alg\otimes\la_{s^{-1}}\la_{t^{-1}})\co{\alg}(am)\bigr)
		\\&=\fmod(s)E_{ts}(am).
\end{align*}
\item[\textup{(iv)}] This follows from (ii) and (iii) together with Proposition~\ref{prop:Fourier coef in espaco spectral}.
\end{enumerate}
\vskip-12pt
\end{proof}

The following result gives some further properties of the Fourier transform.
In particular, we get continuity properties that generalize \cite[Proposition~6.3]{Exel:Unconditional}
and \cite[Proposition~6.3]{Exel:Spectral}.

\begin{proposition} Let $\alg$ be a \dualgcstar{algebra}. If $a\in \M(\alg)_\ii^+$ is a positive integrable element,
then $t\mapsto \tilde E_t(a)\defeq \fmod(t)^{\frac{1}{2}}\cdot E_t(a)$ is a positive-definite function.
In general, $t\mapsto \tilde E_t(a)$ is linear combination of positive-definite functions.
In particular, it is bounded and strictly continuous. Moreover, if $G$ has equivalent uniform structures, 
then the Fourier transform $t\mapsto E_t(a)$ is bounded and strictly-uniformly continuous on $G$, that is,
for each $b\in \alg$, all the expressions $\|E_{ts}(a)b-E_s(a)b\|$, $\|bE_{ts}(a)-bE_s(a)\|$, $\|E_{st}(a)b-E_s(a)b\|$ and
$\|bE_{st}(a)b-bE_s(a)\|$ converge to zero uniformly in $s$ as $t$ converges to $e$ \textup(the identity element of $G$\textup).
\end{proposition}
\begin{proof}
Since $t\mapsto E_t(a)$ is the Fourier transform of the integrable element $\co{\alg}(a)\in \MM{\alg\otimes\csg}_\ii$,
all the assertions follow from Corollaries 5.2 and 5.3 in \cite{Buss:Fourier}.
\end{proof}

Finally, we describe the Fourier transform in the case of a dual coaction.

\begin{theorem}\label{the:FourierTransfDualCoactions}
Let $\fell=\{\fell_t\}_{t\in G}$ be a Fell bundle and consider the \cstar{algebra} $\cstarfell$ with the dual coaction
$\co{\fell}$ of $G$ \textup(defined by Equation~\textup{\eqref{139}}\textup). Identifying $\fell_t$ as a subspace of $\MM{\cstarfell}$, we have
$$E_t(\xi)=\xi(t)\quad \mbox{for all }\xi\in \cont_c(\fell)^2\mbox{ and }t\in G.$$
\end{theorem}
\begin{proof} Using polarization and the same idea as in the proof of Theorem~\ref{the:DualCoactionIntegrable},
this result follows from Proposition~\ref{prop:positive-definite implica integravel}
and Equation~\eqref{eq:FormulaDualCoaction}
\end{proof}

\begin{remark}\label{232} Let us illustrate how it is easy to show Theorem~\ref{the:FourierTransfDualCoactions} if $G$ is discrete.
Moreover, in this case the result holds for all $\xi\in L^1(\fell)$. Indeed, if $G$ is discrete,
$$\co{\fell}(\xi)=\sum\limits_{s\in G}\xi(s)\otimes\la_s$$
for all $\xi\in L^1(\fell)$ and the functional $\f_t$ is bounded
and satisfies $\f_t(\la_s)=\dtg_{t,s}$ for all $t,s\in G$, where $\dtg_{t,s}$ denotes the delta Kronecker's function.
Therefore,
$$E_t(\xi)=(\id\otimes\f_t)\left(\sum\limits_{s\in G}\xi(s)\otimes\la_s\right)=\sum\limits_{s\in G}\xi(s)\f_t(\la_s)=\xi(t).$$
\end{remark}

Theorem~\ref{the:FourierTransfDualCoactions} is a generalization of \cite[Theorem~5.5]{Exel:Unconditional}
to non-abelian groups. Theorem~5.5 in \cite{Exel:Unconditional} is proved using an appropriate Fourier inversion formula.
To explain what we mean, let us assume that $G$ is abelian.
Then the dual coaction $\co{\fell}$ corresponds to the action $\beta$ of $\dualg$ on $\cstarfell$ given by
$$\beta_x(\xi)(t)=\overline{\braket{x}{t}}\xi(t)\quad\mbox{for all } \xi\in \cont_c(\fell), \,\,t\in G,\,\,x\in \dualg.$$
It is easy to see that $\beta_x(b_t)=\overline{\braket{x}{t}}b_t$ for all
$b_t\in \fell_t\sbe \MM{\cstarfell}$. Equation~\eqref{eq:SectionAsIntegral} yields
$$\beta_x(\xi)=\int_G^\st\overline{\braket{x}{s}}\xi(s)\dd{s}=\int_G^\su \overline{\braket{x}{s}}\xi(s)\dd{s}.$$
Thus, we may think of $\beta_x(\xi)$ as a generalized Fourier transform of $\xi\in \cont_c(\fell)$.
In this way, Theorem~\ref{the:FourierTransfDualCoactions} is the Fourier inversion formula:
$$\int_{\dualg}^\su \langle x,t\rangle\left(\int_G^\su \overline{\braket{x}{s}}\xi(s)\dd{s}\right)\dd{x}
			=\xi(t)\quad\mbox{for } t\in G,\,\,\xi\in \cont_c(\fell)^2.$$

Now we prove an analogous version of Theorem~\ref{the:FourierTransfDualCoactions} for the
reduced \cstar{algebra} $\cstarfellr$ using the regular representation
$\la_\fell\colon\cstarfell\to \cstarfellr$ as an equivariant homomorphism. First, we need a preliminary result.

\begin{proposition}\label{pro:CompatibilitySpectralElementsEquivHomom} Let $(\alg,\co{\alg})$ and $(\algb,\co{\algb})$ be coactions of $G$.
Suppose that $\pi\colon\alg\to \M(\algb)$ is a nondegenerate equivariant \Star{}homomorphism.
If $a\in \M(\alg)_\ii$, then $\pi(a)\in \M(\algb)_\ii$ and
$$E_t\bigl(\pi(a)\bigr)=\pi\bigl(E_t(a)\bigr)\quad \mbox{for all }t\in G.$$
\end{proposition}
\begin{proof}
By Proposition~\ref{prop:PermanencePropertiesIntegrableCoactions}, $\pi(a)\in \M(\algb)_\ii$ for all $a\in \M(\alg)_\ii$.
Moreover, the equivariance of $\pi$ yields the desired result:
\begin{align*}
E_t\bigl(\pi(a)\bigr)&=(\id_\algb\otimes\f)\Bigl((1_\algb\otimes\la_t^{-1})\co{\algb}\bigl(\pi(a)\bigr)\Bigr)
		 \\&=(\id_\algb\otimes\f)\Bigl((\pi\otimes\id)\bigl((1_\alg\otimes\la_t^{-1})\co{\alg}(a)\bigr)\Bigr)
		 \\&=\pi\Bigl((\id_\algb\otimes\f)\bigl((1_\alg\otimes\la_t^{-1})\co{\alg}(a)\bigr)\Bigr)
		\\&=\pi\bigl(E_t(a)\bigr).
\end{align*}
\vskip-14pt
\end{proof}

\begin{corollary}\label{039} Let $\fell=\{\fell_t\}_{t\in G}$ be a Fell bundle over $G$ and
consider the dual coaction $\co{\fell}^\red$ of $G$ on $\cstarfellr$.
Then
$$E_t\bigl(\la_\fell(\xi)\bigr)=\la_\fell\bigl(\xi(t)\bigr)\quad \mbox{for all }\xi\in \cont_c(\fell)^2.$$
\end{corollary}
\begin{proof} The assertion follows from Theorem~\ref{the:FourierTransfDualCoactions} and
Proposition~\ref{pro:CompatibilitySpectralElementsEquivHomom}.
\end{proof}

\subsection{The Fourier inversion theorem for integrable coactions}
Let $G$ be a locally compact group, let $\fell$ be a Fell bundle over $G$, and let $a\in \cont_c(\fell)^2$.
Theorem~\ref{the:FourierTransfDualCoactions} allows us to rewrite the formula
$$a=\int_G^\st a(t)\dd{t}=\int_G^{\su}a(t)\dd{t}$$
in the form
\begin{equation}\label{eq:SpectralDecompElement}
a=\int_G^\su E_t(a)\dd{t}.
\end{equation}
Note that the last equation above makes sense for an integrable element
$a$ of an arbitrary \dualgcstar{algebra} $\alg$ provided $t\mapsto E_t(\xi)$ is strictly-unconditionally integrable.
The goal of this section is to investigate when Equation~\ref{eq:SpectralDecompElement} holds in this general setting.
Recall that $\fourst{G}\cong C_\red^*(G)^*$ denotes the Fourier-Stieltjes algebra of $G$.

\begin{lemma}\label{lem:SpectralElemUnderA(G)Action} Let $(\alg,\co{\alg})$ be a \dualgcstar{algebra}.
If $a\in \M(\alg)_\ii$ and $\omega\in \fourst{G}$, then $\omega* a\in \M(\alg)_\ii$ and
$$E_t(\omega* a)=E_t(a)\omega(t)\mbox{ for all } t\in G.$$
\end{lemma}
\begin{proof} The Plancherel weight $\f$ on $\csg$ is both left and right invariant with respect to the comultiplication $\com$,
that is, the quantum group $(\csg,\com)$ is unimodular.
It follows from Lemma 4.11 in \cite{Buss-Meyer:Square-integrable} that $\omega*a\in \M(\alg)_\ii$ and
$E_e(\omega*a)=\Av(\omega*a)=\omega(e)\Av(a)=\omega(e)E_e(a)$.
The desired result for arbitrary $t\in G$ now follows from the following calculation
(where we use again that $\f$ is right invariant):
\begin{align*}
E_t(\omega* a)&=(\id_\alg\otimes \f)\bigl((1_\alg\otimes \la_t^{-1})\co{\alg}(\omega* a)\bigr)\\
&=(\id_\alg\otimes \f)\Bigl((\id_\alg\otimes \id\otimes \omega)\bigl((1_\alg\otimes
                                                \la_t^{-1}\otimes 1)(\id_\alg\otimes\com)\co{\alg}(a)\bigr)\Bigr)
\\&=(\id_\alg\otimes \f)\Bigl((\id_\alg\otimes \id\otimes \omega\la_t)(\id_\alg\otimes\com)
                                                    \bigl((1_\alg\otimes \la_t^{-1})\co{\alg}(a)\bigr)\Bigr)
\\&=(\id_\alg\otimes \omega\la_t)\Bigl((\id_\alg\otimes \f \otimes \id)\bigl((\id_\alg\otimes\com)
                                                        ((1_\alg\otimes \la_t^{-1})\co{\alg}(a))\bigr)\Bigr)
           \\&=(\id_\alg\otimes \omega\la_t)\Bigl((\id_\alg\otimes\f)\bigl((1_\alg\otimes\la_t^{-1})\co{\alg}(a)\bigr)\otimes 1\Bigr)
				\\&=E_t(a)\omega(t).
\end{align*}
\vskip-14pt
\end{proof}

\begin{theorem}[Fourier's inversion Theorem]\label{theo:FourierInversionTheoremForIntegrableCoactions}
Let $(\alg,\co{\alg})$ be a \dualgcstar{algebra}. Let $a\in \alg_\ii$ and suppose that the Fourier transform
$G\ni t\mapsto E_t(a)\in \M(\alg)$ is strictly-unconditionally integrable.
Then
$$\co{\alg}(a)=\int_G^\su E_t(a)\otimes\la_t\dd{t}.$$
If $\co{\alg}$ is injective, then $\int_G^\su E_t(a)\dd{t}=a$.
In general, we have
$$\int_G^\su E_t(\omega* a)\dd{t}=\omega* a\quad\mbox{for all }\omega\in \fourst{G}.$$
\end{theorem}
\begin{proof} Since the function $t\mapsto E_t(a)$ is strictly-unconditionally integrable, so is the function
$t\mapsto E_t(a)\otimes\la_t=\co{\alg}\bigl(E_t(a)\bigr)$, and
$$\co{\alg}\left(\int_G^{\su}E_t(a)\dd{t}\right)=\int_G^{\su}E_t(a)\otimes\la_t\dd{t}.$$
Proposition~\ref{the:FourierInversionTheorem} yields the first assertion:
$$\co{\alg}(a)=\int_G^\su E_t(a)\otimes\la_t\dd{t}.$$
This implies
$$\co{\alg}\left(\int_G^{\su} E_t(a)\dd{t}\right)=\int_G^\su E_t(a)\otimes \la_t\dd{t}=\co{\alg}(a).$$
Therefore, if $\co{\alg}$ is injective, then $\int_G^\su\!E_t(a)\dd{t}=a$.
Finally, if $\omega\in \fourst{G}$, then Lemma~\ref{lem:SpectralElemUnderA(G)Action} yields
\begin{align*}
\omega* a &=(\id_\alg\otimes\omega)\bigl(\co{\alg}(a)\bigr)
           \\ &=(\id_\alg\otimes\omega)\left(\int_G^\su E_t(a)\otimes\la_t\dd{t}\right)
           \\ &=\int_G^\su E_t(a)\omega(t)\dd{t}
           \\ &=\int_G^\su E_t(\omega* a)\dd{t}.
\end{align*}
\vskip-14pt
\end{proof}

\begin{remark}\label{rem:InjectivityNecessary}
The injectivity of $\co{\alg}$ in Theorem~\ref{theo:FourierInversionTheoremForIntegrableCoactions}
is really necessary. In fact, if $a\in \ker(\co{\alg})$,
then $a\in \alg_\ii$ and $E_t(a)=0$ for all $t\in G$.
Thus, if $\co{\alg}$ is not injective, and if $0\not=a\in \ker(\co{\alg})$, then
$$\int_G^{\su} E_t(a)\dd{t}=0\not=a.$$
\end{remark}

Theorem~\ref{theo:FourierInversionTheoremForIntegrableCoactions} generalizes Proposition~6.6 in \cite{Exel:Spectral} to non-abelian groups.
Assume that $G$ is abelian. Then, under the usual identification
$\M\bigl(\alg\otimes C_\red^*(G)\bigr)\cong \cont_b\bigl(\dualg,\M^{\st}(\alg)\bigr)$, the element $E_t(a)\otimes\la_t$ corresponds
to the function $x\mapsto \overline{\braket{x}{t}}E_t(a)$. Hence, Theorem~\ref{theo:FourierInversionTheoremForIntegrableCoactions} yields
$$\int_G^\su\overline{\braket{x}{t}}E_t(a)\dd{t}=\alpha_x(a),$$
where $\alpha$ is the action of $\dualg$ on $\alg$ corresponding to the coaction $\co{\alg}$.
The Fourier coefficient $E_t(a)$ in this case is given by the integral
$\int_{\dualg}^\su\braket{x}{t}\alpha_x(a)\dd{x}$. Thus, we may rewrite the above equation in the form of a
generalized Fourier inversion formula:
$$\int_G^\su\overline{\braket{x}{t}}\left(\int_{\dualg}^\su\braket{x}{t}\alpha_x(a)\dd{x}\right)\dd{t}=\alpha_x(a).$$

\section{Summary and outlook}

In the previous sections we saw that Fell bundles over locally compact groups
give rise to examples of integrable coactions -- the dual coactions.
It is therefore natural to ask whether every integrable coaction comes from a Fell bundle in this way, that is,
if every integrable coaction is isomorphic to a dual coaction. Unfortunately, this is
a delicate question which is not true in general. The question was initially proposed by Ruy Exel in \cite{Exel:Unconditional}
for the case of abelian groups. In a subsequent article \cite{Exel:Spectral}, Exel gave a partial solution to the problem (for abelian groups).
He proved that under certain additional conditions on an integrable coaction of $G$,
it is in fact possible to construct a Fell bundle over $G$ in such a way that the initial coaction is isomorphic to the dual coaction
of the constructed Fell bundle.

As already mentioned, if $G$ is abelian, coactions of $G$ correspond to actions of the dual group $\dualg$ and, therefore, the theory
of coactions is not necessary in this case. The main goal of Exel in \cite{Exel:Spectral} was to find
conditions on a given action of $\dualg$ that guarantee it is equivalent to a dual action on the \cstar{algebra} of some
Fell bundle over $G$. The \emph{spectral theory} developed by Exel gives enough conditions
on the spectral elements $E_t(a)$ of an integrable action for that goal to be achieved.
More precisely, it is required the existence of a dense subspace $\W$ consisting of integrable elements
that are \emph{relatively continuous}, meaning that
$$\|E_{ts}(a)E_{r}(b)-E_t(a)E_{sr}(b)\|\to 0,\quad\mbox{ uniformly in }t,r\in G\mbox{ as }s\to e.$$
for all $a,b\in \W$. Although technical, this condition is naturally established in the case of the dual action of $\dualg$ on $C^*(\fell)$,
where $\fell$ is a Fell bundle over $G$.
Moreover, the main result in \cite{Exel:Spectral} says that an action of $\dualg$ is isomorphic to a
dual action if and only if it is \emph{continuously integrable}, that is, integrable
with a relatively continuous dense subspace $\W$.


The spectral theory for group actions on \cstar{algebras} developed by Exel only makes sense for abelian groups.
However, our work allows us to extend that theory to non-abelian groups just replacing actions by coactions of groups.
In fact, besides of defining integrable coactions and proving that dual coactions belong to this class,
we developed a theory of spectral elements $E_t(a)$ that makes sense for any integrable element $a$
in a \dualgcstar{algebra} $A$, where $G$ is an arbitrary locally compact group.
Note that relative continuity is a concept that still makes sense for coactions and,
therefore, it is licit to speak of \emph{continuously integrable coactions}.
In fact, with ideas similar those of Ruy Exel in \cite{Exel:Spectral}, it is possible to characterize dual coactions
through continuous integrability. One of the main tools in this direction is the Fourier inversion Theorem for integrable coactions
which is also one of our main results (Theorem~\ref{theo:FourierInversionTheoremForIntegrableCoactions}).
We plan to publish the details of this construction in a future work.


\end{document}